\documentclass[12pt]{amsart}
\usepackage[utf8]{inputenc}
\usepackage{lmodern}
\usepackage{amssymb,amsmath,amsthm,fullpage}
\usepackage{pdfsync,xcolor}
\usepackage{enumerate}
\usepackage{pdfsync}
\usepackage{mathrsfs}
\usepackage[bookmarksnumbered,colorlinks,plainpages,backref]{hyperref}
\usepackage[normalem]{ulem}

\numberwithin{equation}{section}

\newcommand{\abs}[1]{\left\vert#1\right\vert}

\newcommand{\re}{\mathrm{Re}}
\newcommand{\im}{\mathrm{Im}}
\renewcommand{\L}{\mathcal{L}}

\newcommand{\ovl}[1]{\overline{#1}}

\newcommand{\R}{\mathbb R}

\newcommand{\N}{\mathbb N}
\newcommand{\C}{\mathbb C}

\DeclareMathOperator{\supp}{supp}

\DeclareMathOperator{\Ker}{Ker}
\DeclareMathOperator{\Ran}{Ran}

\DeclareMathOperator{\diag}{diag}
\newcommand{\p}{\partial}

\newcommand{\dbar}{\bar\partial}

\newcommand{\dbarb}{\bar\partial_b}
\newcommand{\dbarbs}{\bar\partial^*_b}

\newcommand{\vp}{\varphi}

\newcommand{\nn}{\nonumber}

\newcommand{\wg}{\wedge}
\newcommand{\I}{\mathcal{I}}

\DeclareMathOperator{\Tr}{Tr}

\newcommand{\la}{\langle}
\newcommand{\ra}{\rangle}

\newtheorem{thm}{Theorem}[section]
\newtheorem{prop}[thm]{Proposition}
\newtheorem{proposition}[thm]{Proposition}

\newtheorem{cor}[thm]{Corollary}

\newtheorem*{theorem*}{Teorema}

\theoremstyle{definition}

\newtheorem{definition}[thm]{Definition}

\theoremstyle{remark}

\newtheorem{remark}[thm]{Remark}

\newcommand{\lla}[1]{\left\{ #1\right\} }

\newcommand{\bkj}{b^{\bar{k}j}}

\newcommand{\pare}[1]{\left( #1\right)}
\newcommand{\corc}[1]{\left[ #1\right]}
\newcommand{\ip}[1]{ \left\langle #1 \right\rangle }
\newcommand{\ba}[1]{ \overline{#1}}

\newcommand{\vab}[1]{\left| #1 \right| }

\newcommand{\norm}[1]{{\| {#1} \|}}
\newcommand{\nnorm}[1]{{\left\| {#1} \right\|}}
\newcommand{\Norm}[1]{{\left\| {#1} \right\|}}

\newcommand{\ve}{\varepsilon}

\newcommand{\lb}{\bar{L}}
\newcommand{\lba}{\bar{L}^{*}}

\allowdisplaybreaks

\begin{document}

\title{Subelliptic and Maximal $ L^p $ Estimates for {the} Complex Green Operator on non-pseudoconvex domains}%
\author{Coacalle J.}
\address{Departamento de Computacão e Matemática, Universidade de São Paulo, Ribeirão Preto,	SP, 14040-901, Brasil}
\email{jcoacalle@usp.br}

\thanks{The author was supported by the São Paulo Research Foundation
	(FAPESP - grant 22/02211-8) 
}

\maketitle

\vspace{-0.8cm}

\begin{abstract} 
We {prove} subelliptic estimates for {the} complex Green operator $ K_q $ at a specific level $ q $ of the $ \dbarb $-complex, defined on a {not necessarily pseudoconvex} CR manifold satisfying the commutator finite type condition. Additionally, we {obtain} maximal $ L^p $ estimates for $ K_q $ by considering closed-range estimates. Our results apply to a family of manifolds that includes a class of weak $ Y(q) $ manifolds satisfying the condition $ D(q) $. We employ {a} microlocal decomposition and Calderon-Zygmund theory to obtain subelliptic and maximal-$ L^p $ estimates, respectively.	
\end{abstract}

\section{Introduction}

A fundamental problem in the study of Kohn-Laplacian, $\Box_b$, is determining under which conditions its inverse, the complex Green operator $K_q$, has a gain in regularity.  
This is because, despite being of the {Laplace} type, $\Box_q$ generates a non elliptic second order system of partial differential equations. 
The problem {has been} widely explored for pseudoconvex and $ Y(q) $ manifolds, and we refer the reader to \cite{KoNi1965,RothStein76,GriRoth88,Derridj91,koenig02} for {back-ground}. We introduce a subtle comparability condition well adapted to a kind of manifold with the finite type property, more general than those mentioned above. This condition is still within the scope of the well-known condition $ D(q) $, and we use it to show a gain in regularity for $ K_q $ through $\ve$-subelliptic estimates in $L^2$-Sobolev spaces. 
{Furthermore, we show optimal gain in regularity for $ K_q $ at some tangential directions} through maximal $ L^p $ estimates on a specific level of the $ \dbarb $-complex where we assume a closed range hypothesis.

Let $ M $ be a {smooth} compact orientable  CR manifold of hypersurface type embedded in $ \C^{N} $, of real dimension $ 2n-1 $.
The CR structure {on} $ M $ allows us to define the tangential Cauchy Riemann operator $ \dbarb $, which acts on the $ (0,q) $-form spaces in the $ \dbarb $-complex.
The Kohn-Laplacian operator $ \Box_b $ is defined at each form level on $ M $ by $ \Box_b=\dbarb\dbarbs+\dbarbs\dbarb $, where $ \dbarbs $ is {the} $ L^2 $-adjoint of $ \dbarb $. When $ \Box_b $ {has a continuous inverse} on $ {^\perp \ker \Box_b} $, {it} is called by the complex Green operator $ K_q $.
The existence and the regularity properties of $ K_q $ are achieved by obtaining closed range estimates in $ L^2 $ and $ L^2 $-Sobolev spaces. Establishing these estimates relies strongly on the geometry of the manifold $ M $, and the results may be valid for only specific levels of the $ \dbarb $-complex.
{Examples of well-known geometric conditions on $ M $ are} 
strong and weak pseudoconvexity, $ Y(q) $ and weak-$ Y(q) $. This latter was recently introduced by Harrington and Raich
\cite{HaRa11,HaRa15}, 
{and it is the most general known sufficient condition for continuity of $ K_q $ on $ L^2 $-Sobolev spaces.
These geometric conditions are defined in terms of the Levi form.}
Essentially, the {Levi form} measures the non-involutivity of the CR structure and its conjugate, and associates a Hermitian matrix {to} it.
We say $ \ve $-subelliptic estimates are satisfied, if
\begin{equation}\label{eq:main}
	\norm{u}_\ve^2 \lesssim  \norm{\dbarb u}^2_{L^2}+\norm{\dbarbs u}^2_{L^2} +\norm{u}^2_{L^2}  
\end{equation}
for every smooth $ (0,q) $-form $ u $,
where $ \norm{u}_\ve $ denotes the $ L^2 $-Sobolev norm of $ u $ of order $ \ve>0 $. 
Also, if $ U $ is an open set in $ M $ and $ \lla{X_j}_{1\leq j\leq 2n-2} $ is a set of real vector fields on $ U $ such that $ \lla{X_j+iX_{j+n-1}}_j $ spans the CR structure of $ M $, we say that maximal $ L^p $ estimates are satisfied at $ U $ when
\begin{equation}\label{ec:max}
	\sum_{i,j=1}^{2n-2}\norm{ X_iX_j u }_{L^p}\lesssim \norm{\Box_b u}_{L^p}+\norm{u}_{L^p}
\end{equation}
for every smooth $ (0,q) $-form $ u $ supported in $ U $, and for every $ p\in (1,\infty) $.

In \cite{Kohn1965}, Kohn proved {that $ Y(q) $ implies that \eqref{eq:main} holds with $ \ve=1/2 $.}
In particular $ Y(q) $ implies that the CR manifold has the finite commutator type property of order $ 2 $. Also {$ M $} satisfies a comparable eigenvalues condition on the Levi form called $ D(q) $. Roughly speaking, $ D(q) $ locally compares the sum of any $ q $ eigenvalues with the trace of the Levi matrix.
In \cite{Derridj91}, Derridj used microlocal analysis to show subelliptic estimates {hold for} $ (0,1) $-forms on pseudoconvex hypersurfaces in $ \C^n $, assuming $ D(1) $ and finite commutator type condition of some order. Koening generalized Derridj's result for higher levels assuming $ D(q) $ (see \cite[Appendix B]{koenig02}).

On the other hand, the work done in \cite{Kohn1965}, allow us to conclude that under the $ Y(q) $ condition, $ K_q $ gains just a derivative at all tangential directions. However, Rothschild and Stein in \cite{RothStein76} showed that
\eqref{ec:max} {holds},
which means that $ K_q $ gains two derivatives in the ``good''  tangential directions $ \lla{X_j}_{j=1}^{2n-2} $. 
In \cite{GriRoth88}, Grigis {and} Rothschild showed that $ D(q) $ is a necessary condition to obtain the maximal estimates \eqref{ec:max}.
For the pseudoconvex case, 
{Koening showed that the hypothesis of closed range estimates, $ D(q) $, and finite commutator type conditions suffice to prove \eqref{ec:max} \cite{koenig02}.}
It is worth mentioning that Derridj \cite{Derridj78} and Ben Moussa \cite{Ben00} studied analogous estimates for the $\dbar$-Neumann problem in 
pseudoconvex domains,
while Harrington and Raich recently examined this for non pseudoconvex domains in \cite{HaRa22}.

Our results about subelliptic estimates {apply to} certain non-pseudoconvex manifolds, which satisfy $ D(q) $ and the finite type condition. These manifolds satisfy a comparability condition introduced in this paper (see Definition \ref{def:comp}). This new condition allows us {adapt} conveniently the procedure in \cite{Derridj91} and obtain subelliptic estimates of the type \eqref{eq:main}, Theorem \ref{thm:subell}. We prove the maximal $L^p$ estimates, Proposition \ref{prop:maximalestimates} , in the spirit of \cite{koenig02,christ88}, using previous subelliptic estimates together with closed range hypothesis just at some specified level of $(0,q)$-forms. Due to this last hypothesis and the fact that $ M $ may not be pseudoconvex, we redo the process performed in \cite{koenig02} because {our hypotheses do not imply that} $ K_{q-1} $ or $ K_{q+1} $ exist (see \cite[Proposition 5.5]{koenig02} and subsequent remark).

We distribute the topics as follows. Section $ 2 $ presents the setting in of our CR manifold, the definition of our comparability condition, and its relation with the well-known $ D(q) $. Section $ 3 $ contains a special integration by parts and the microlocal configuration to obtain $ \ve $-subelliptic estimates. In Section $ 4 $, we utilize the subelliptic estimates and an escalation method. We provide pointwise and differential estimates for the distribution kernels associated with $ K_q $, along with cancellation properties on its action on bump functions. Then, we demonstrate maximal $ L^p $ estimates using a Calderon-Zygmund theory.

\section{Preliminaries}

\subsection{The CR structure on $M$}
Let $M\subset\C^N$ be a smooth, compact, orientable CR manifold of hypersurface type of real dimension $ 2n-1$. Since $M$ is embedded, we take the CR structure on $M$ to be the induced CR structure $T^{1,0}(M)=T^{1,0}(\C^N)\cap T(M)\otimes \C$. Let $T^{p,q}(M)$ denote the exterior algebra generated by $T^{1,0}(M)$ and $T^{0,1}(M)$ and $\Lambda^{p,q}(M)$ the bundle of $(p,q)$-forms on $T^{p,q}(M)$. The calculations in this work do not depend on $ p $ because $M \subset \C^N$, and so we will always assume $ p=0 $. 

\subsection{$\dbarb$ on embedded manifolds} 
We use the induced pointwise inner product on $ \Lambda^{0,q}(M) $, $ \ip{\cdot,\cdot}_x $, coming from the inner product on $ (0,q) $-forms in $ \C^N $. Therefore, we define an $ L^2 $-inner product on $\Lambda^{0,q}(M)$ given by
\[
(\vp,\psi) =\int_M \la \vp,\psi \ra_x \, dV(x)
\]
where $dV$ is the volume element on $M$. We use $ \|\cdot\|_{L^2_q(M)} $ to denote the associated $L^2$-norm of the inner product above and $ L^2_q(M) $ to the completion of $\Lambda^{0,q}(M)$ with respect to $ \|\cdot\|_{L^2_q(M)} $. The tangential Cauchy-Riemann operator $\dbarb$ is the restriction of the De Rham exterior derivative $d$ to $\Lambda^{0,q}(M)$.  Additionally, we are denoting (by an abuse of notation) with $ \dbarb $ the closure of $ \dbarb $ in norm $ \|\cdot\|_{L^2_q(M)} $. In this way,  $ \dbarb :L^2_q(M)\rightarrow L^2_{q+1}(M) $ is a well-defined, closed, densely defined operator with $L^2$ adjoint $\dbarbs: L^2_{q+1}(M)\rightarrow L^2_q(M)$. We define the Kohn Laplacian $ \Box_b:L^2_q(M)\rightarrow L^2_q(M) $ by
\[
\Box_b:=\dbarbs\dbarb +\dbarb \dbarbs.
\]
Using a covering $ \lla{U_\mu}_\mu $ and a partition of unity $ \lla{\zeta_\mu}_\mu $ subordinate to the covering, we define the Sobolev norm of order $ s $, of a $ (0,q) $-form $ f $, given by 
\[
\norm{f}_{H^s_{q}}^2=\sum_\mu \norm{\tilde{\zeta}_\mu \Lambda^s \zeta_\mu f^\mu}^2_{L^2_q(M)}
\]
where $ f^\mu $ is the form $ f $ expressed in local coordinates on $ U_\mu $, $ \Lambda $ is the pseudodifferential operator with symbol $ (1+\abs{\xi}^2)^{\frac{1}{2}} $ and $ \tilde{\zeta}_\mu$ is a cutoff function that dominates $ \zeta_\mu $ such that $ \supp{\tilde{\zeta}_\mu}\subset U_\mu $.

\subsection{The Levi form and the weak $ Y(q) $ condition}
By the orientability of $ M $, we choose a purely imaginary unit vector $ T $ {that is} orthogonal to $T^{1,0}(M)\oplus T^{0,1}(M)$.
Let $\gamma$ be the globally defined $1$-form that annihilates $T^{1,0}(M)\oplus T^{0,1}(M)$ and is normalized so that $\la \gamma, T \ra = -1 $.

	The Levi form at a point $x\in M$ is the Hermitian form given by $\la {d\gamma_x ,L\wedge \bar{L}'} \ra$, where $L,L'\in T^{1,0}_x(U)$ and $U$ is a small neighborhood of $x\in M$.

If $\{L_{j}\}_{j=1}^{n-1}$ is a local basis of $T^{1,0}(U)$ and $ c_{jk} $ are such that $ \la {d\gamma,L_j\wedge\ba{L}_k}\ra=c_{jk} $, then $ [c_{jk}] $ is called by the Levi matrix with respect to $\{L_{j}\}_{j=1}^{n-1}$ and $T$.

We say that $M$ is pseudoconvex (resp., strictly pseudoconvex) at a point $ p\in M $ if all the eigenvalues of the Levi matrix are nonnegative (resp., positive) at $ p $. We say $ M $ is pseudoconvex or strictly pseudoconvex if it is at {every} point.
Also, for fixed $ 1\leq q \leq n-1 $, we say $ M $ satisfies the $ Z(q) $ condition at $ p $
if the Levi matrix has at least $ n-q $ positive eigenvalues or at least $ q+1 $ negative eigenvalues at $ p $. $ M $ satisfies the $ Y(q) $ condition at $ p $ if both $ Z(q) $ and $ Z(n-1-q) $ conditions are satisfied at $ p $. We say $ Z(q) $ or $ Y(q) $ condition holds {on} $ M $ {if does at every} point.

Due to \cite{Kohn1965,Kohn86,Nicoara2006}, the geometric conditions mentioned in the previous paragraph are sufficient to show a closed range of $ \dbarb $ on $ L^2_q(M) $. However, another weaker geometric condition was introduced in \cite{HaRa11,HaRa15,HaPeRa15,CoRa19} {also suffices to establish closed range at level $ q $}, which we define as follows.

\begin{definition}\label{defn:weak Z(q)}
	For $1\leq q\leq n-1$, we say $M$ satisfies weak $Z(q)$ if there exists a real $\Upsilon\in T^{1,1}(M)$ satisfying:
	\begin{enumerate}[(A)]
		\item\label{wzq1} $\vab{\theta}^2\geq (i\theta\wedge\ba{\theta})(\Upsilon)\geq 0$ for all $\theta\in \Lambda^{1,0}(M)$ 
		\item $\mu_1+\mu_2+{\cdots}+\mu_{q}-i\la{d\gamma_x,\Upsilon}\ra\geq 0$ where $\mu_1,...,\mu_{n-1}$ are the eigenvalues of the Levi form at $x$
		in increasing order.  \label{wzq2}
		\item\label{eqn:omega(Upsilon) neq q}
		$\omega(\Upsilon)\neq q$ where $\omega$ is the $(1,1)$-form associated to the induced metric on $\C T(M)$.
	\end{enumerate}
	We say that $M$ satisfies weak $Y(q)$ if $M$ satisfies both weak $Z(q)$ and weak $Z(n-q-1)$.
\end{definition}

\subsection{Finite Type condition and its consequences}
Let $ U $ be an open set in $ M $ and $ \lla{L_j}_{1\leq j\leq n-1} $ an orthonormal basis of $ T^{1,0}(M) $ on $ U $. Let $ X_j={\re L_j} $, $ X_{j+n-1}={\im L_j} $, for $ 1\leq j\leq n-1 $ and
\[
\L_m = \lla{ [ X_{j_1},[X_{j_2},[ \cdots ,[X_{j_{k-1}},X_{j_k}] \cdots ] ] ]  : k\leq m, X_{j_s}\in \lla{ X_k }_{1\leq k\leq 2n-2} }.
\]
Observe that
	\begin{equation} \label{eq:horobs}
	\sum_{j=1}^{2n-2} \norm{X_ju}^2_{L^2} \simeq \sum_{j=1}^{n-1} \pare{\norm{L_ju}^2_{L^2} + \norm{\bar{L}_ju}^2_{L^2} },
	\end{equation}
for any smooth function $ u $ supported in $ U $. 
We will say that the finite commutator type (or finite type) condition is satisfied at $ U $ if there exists a vector field system $ \lla{L_j}_{1\leq j\leq n-1} $ defined on $ U \subset M$ such that $ \operatorname{span}{ \L_m} = T(U) $ for some $ m\in \lla{2,3,\cdots} $.The order of the finite type {is the smallest such} $ m $. 
We say that $ M $ has the finite type property if for each point {in} $ M ${, there} exists an open neighborhood $ U $ {on which} the finite type condition is satisfied.
It is a known result due to Hormander that if the finite type condition of order $ m $ is satisfied at $ U $,
	then there exists $ \varepsilon>0 $ and $ C>0 $ such that are satisfied the next subelliptic estimates
	\begin{equation}\label{thm:horm}
	\norm{u}_{H^\ve}^2\leq C \pare{\sum_{j=1}^{2n-2} \norm{X_ju}^2_{L^2}+\norm{u}^2_{L^2} }, \qquad \forall u\in C^\infty_0(U).
	\end{equation}
Also, $ \varepsilon=\frac{1}{m} $ is the best constant here {(see e.g. Appendix in \cite{tanaka75})}. In the following, we assume $ M $ has the finite type property. 

\subsection{The comparable weak $ Y(q) $ condition}

\begin{definition}\label{def:comp}
	Let $ U $ be an open set in $ M $. For $ 1\leq q\leq n-1 $, we say that comparable weak $ Z(q) $ is satisfied on $ U $ if there exists a real bivector $ \Upsilon \in T^{1,1}(U) $ satisfying:
	\begin{enumerate}[(a)]
		\item\label{def:cwza} $ |\theta|^2 > i(\theta \wg \bar{\theta})\Upsilon >0 $ for all $ \theta \in \Lambda^{1,0}(U) $
		\item\label{def:cwzb} $ \sigma_q -i\ip{d\gamma , \Upsilon} \geq 0 $, where $ \sigma_q $ represents the sum of {the smallest} $ q $ eigenvalues of the Levi matrix.
	\end{enumerate}
\end{definition}

We say that {comparable weak $ Y(q) $} is satisfied on $ U $ ({resp.,} $ M $) if it {satisfies} both comparable weak $ Z(q) $ and  comparable weak $ Z(n-1-q) $ on $ U $ ({resp., if every point of $ M $ has a neighborhood $ U $ on which comparable weak $ Y(q) $ holds}. 
Weak $ Z(q) $ manifolds satisfying \eqref{def:cwza} in Definition \ref{def:comp} instead of \eqref{wzq1} in Definition \ref{defn:weak Z(q)} satisfy comparable weak $ Z(q) $. 
In essence, the existence of the bivector in Definition \ref{def:comp} refers to the existence of a Hermitian matrix $ \Upsilon $ with smooth entries, eigenvalues in $ ]0,1[ $, and such that $ \sigma_q-\Tr(\Upsilon C) \geq 0 $ where $ C $ is the Levi matrix. 
Thus, we refer to $ \Upsilon $ as a matrix instead of a bivector.
In \cite{HaRa22} is introduced an analogous version of this condition.
Along with other regularity and boundary conditions in $ \Upsilon $, they show local maximal $ L^2 $ estimates for the $ \dbar $-Neumann problem on boundaries of domains in $ \C^N $ \cite[Proposition 2.4]{HaRa22}. 

{Indeed}, if $ Z(q) $ is satisfied at $ x\in M $, then a neighborhood $ V_x $ of $ x $ exists such that the comparable weak $ Z(q) $ condition is satisfied at $ V_x $. Let $ C $ denote the Levi matrix, $ C_x $ be the matrix $ C $ evaluated at $ x $, $ \lla{\lambda_j(C)}_j $ be the eigenvalues of the matrix $ C $ in increasing order, and $ A $ be a unitary matrix (with constant entries) such that
$
C_x=A^*DA
$,
where $ D=\diag\lla{\lambda_1(C_x),\cdots, \lambda_{n-1}(C_x)} $. Assume that there exist at least $ n-q $ positive eigenvalues at $ x $. The proof when there exist at least $ q+1 $ negative eigenvalues is similar. Then $ \lambda_q(C_x)>0 $. Let $ \ve $ be a positive constant small sufficient such that 
\begin{equation}\label{eq:yqImpcomp0}
	(1-\ve)\lambda_q(C_x) +\ve( \lambda_1(C_x)+\cdots+\lambda_{q-1}(C_x)- \lambda_{q+1}(C_x)-\cdots - \lambda_{n-1}(C_x) )>0.
\end{equation}
Let $ B=C-A^* DA $. By the continuity of the entries $ B^{ij} $ of $ B $ and since $ B_x=0 $; for $ \delta>0 $, a neighborhood $ V_x $ of $ x $ exists, such that
$
\abs{B_{y}^{ij}}<\delta,
$
for all $ 1\leq i,j\leq n-1$ and  $ y\in V_x $.
Then, in particular, by applying Schur's inequality, we have
\begin{equation}\label{eq:orderdelta}
	\abs{ \lambda_1(B_y) }^2\leq \sum_{i,j=1}^{n-1}\abs{B_y^{ij}}^2=O(\delta^2).
\end{equation}
Define $ \Upsilon=A^* \diag\lla{ 1-\ve,\cdots,1-\ve,\ve,\cdots,\ve }A $, where $ 1-\ve $ appears $ q-1 $ times in the diagonal matrix. So, from Weyl's inequalities (see \cite[Theorem 4.3.1]{HornJohnson13}), we have
\begin{align}
	\lambda_1(C_y)&+\cdots+\lambda_q(C_y)-\Tr( \Upsilon C_y )= \nn \\
	&=\lambda_1(A^*DA+B_y)+\cdots + \lambda_{q}(A^*DA+B_y)
	-\Tr( \Upsilon A^* DA )-\Tr( \Upsilon B_y ) \nn \\
	&\geq \lambda_1(A^*DA)+\cdots +\lambda_q(A^*DA) +q\lambda_1(B_y)
	- \Tr( A\Upsilon A^* D )-\Tr(\Upsilon B_y) \nn \\
	&=(1-\ve)\lambda_q(C_x) +\ve( \lambda_1(C_x)+\cdots+\lambda_{q-1}(C_x)- \lambda_{q+1}(C_x)-\cdots - \lambda_{n-1}(C_x) ) \nn \\
	&\quad + q\lambda_1(B_y)-\Tr(\Upsilon B_y). \label{eq:yqImpcomp}
\end{align}
Now, since $ \Tr(\Upsilon B_y) $ is a linear combination of elements $ B_y^{ij} $, we have
\begin{equation}\label{eq:orderdelta2}
	\Tr(\Upsilon B_y) = O(\delta).
\end{equation}
So, from \eqref{eq:orderdelta}, \eqref{eq:orderdelta2}, and \eqref{eq:yqImpcomp0}, we can choose $ \delta $ small sufficient such that the right-hand side in \eqref{eq:yqImpcomp} be positive for all $ y $ in some neighborhood $ V_x $ of $ x $. This shows that comparable weak $ Z(q) $ condition is satisfied at $ V_x $ using $ \Upsilon $ as we defined before.

\subsection{Relation between comparable weak $ Y(q) $ and $ D(q) $ conditions.}
In this section, we expose the relation between our comparable condition and the well-known condition $ D(q) $, given in \cite{GriRoth88} and {inspired by} \cite{Derridj78}. 
It is worth mentioning that {the motivation of} $ D(q) $
and our condition {is} to generalize the condition $ Y(q) $, which  is equivalent to 
\begin{equation}\label{eq:equivyq}
 -\sigma^-<\sigma_q<\sigma^+ ,
\end{equation}
where $ \sigma_q $ {represents} any possible sum of $ q $ eigenvalues of the Levi matrix $ [c_{jk}] $ and
\[
\sigma^++\sigma^-=\sum_{j=1}^{n-1}\abs{\mu_j} \ , \qquad \sigma^+-\sigma^-=\sum_{j=1}^{n-1}\mu_j,
\]
with $ \lla{\mu_j}_{j=1}^{n-1} $ being the eigenvalues of $ C=[c_{jk}] $ (see \cite{GriRoth88}). We say that $ M $ satisfies $ D(q) $ in $ U $ if, for any compact $ K\subset U $, there exists a constant $ \ve_K >0 $ such that
\begin{equation}\label{def:Dq}
-\sigma^-+\ve_K(\sigma^++\sigma^-) \leq \sigma_q\leq \sigma^+-\ve_K(\sigma^++\sigma^-).
\end{equation}
On the other hand, the comparable weak $ Y(q) $ condition generalizes \eqref{eq:equivyq} as 
\begin{equation*}
	\Tr(\Upsilon_1 C)\leq \sigma_q\leq \Tr([Id-\Upsilon_2] C) 
\end{equation*}
where $ \Upsilon_1 $ and $ \Upsilon_2 $ are from comparable weak $ Z(q) $ and $ Z(n-1-q) $, respectively.

\begin{remark}\label{rem:cwzqDq}
	If comparable weak $ Y(q) $ is satisfied on some open set $ U $ then $ D(q) $ is satisfied in $ U $. In fact, let $ \Upsilon_1=[b^{rs}] $ and $ \Upsilon_2=[B^{rs}] $ {satisfy} the comparable weak $ Z(q) $ and comparable weak $ Z(n-1-q) $, respectively. In particular we have
	\begin{align*}
			\sigma_q \geq \sum_{r,s=1}^{n-1} b^{rs}c_{sr} = \Tr(\Upsilon_1 [c_{jk}]) \quad,\qquad 
			\sigma_{n-1-q} \geq \sum_{r,s=1}^{n-1} B^{rs}c_{sr} = \Tr (\Upsilon_2 [c_{jk}]).
	\end{align*}
	Let $ \ve_1 $ and $ 1-\ve_2 $ be the smallest and largest eigenvalue of $ \Upsilon_1 $, respectively; $ A $ be a unitary matrix such that $ A^*CA =\diag\lla{\mu_1,\cdots ,\mu_{n-1}} $ in non-decreasing order ($ A $ exists at any point, but it can vary discontinuously in $ U $); $ \sigma^- =-\sum_{j=1}^{t} \mu_{jj}$, $ \sigma^+ = \sum_{j=t+1}^{n-1} \mu_{jj} $, and $ \lla{d_{11},\cdots ,d_{n-1\ n-1}} $ are the diagonal elements of $ A^*\Upsilon_1 A $. 
	Then 
	\begin{align*}
		\Tr(\Upsilon_1 C)&=\Tr( A^* \Upsilon_1 A A^*C A)  
		\geq \max_{1\leq j\leq n-1}\lla{d_{jj}} \sum_{j=1}^{t}\mu_{jj}+\min_{1\leq j\leq n-1}\lla{d_{jj}} \sum_{j=t+1}^{n-1} \mu_{jj}\\
		& \geq (1-\ve_2) \sum_{j=1}^{t}\mu_{jj}+ \ve_1\sum_{j=t+1}^{n-1} \mu_{jj} 
		\geq -\sigma^- +\ve(\sigma^+ +\sigma^-) 
	\end{align*}
	for $ \ve=\min\lla{\ve_1,\ve_2} $.  The penultimate inequality above follows by Schur-Horn Theorem (see \cite[Theorem 4.3.45]{HornJohnson13}). So we obtain the left inequality in \eqref{def:Dq}. Similarly, we obtain
	\begin{align*}
		\Tr(\Upsilon_2 C)& \geq -\sigma^- +\ve(\sigma^+ +\sigma^-),
	\end{align*}
	for some $ 0<\ve<1 $. Then
	\begin{align*}
	\sigma_{n-1-q}\geq -\sigma^- +\ve(\sigma^+ +\sigma^-) &\Longrightarrow \quad \Tr(C)+ \sigma^- -\ve(\sigma^+ +\sigma^-) \geq  \Tr(C)-\sigma_{n-1-q} \\
	& \Longrightarrow \quad \sigma^+ -\ve(\sigma^+ +\sigma^-) \geq \Tr(C)-\sigma_{n-1-q},
	\end{align*}
	and since $ \Tr(C)-\sigma_{n-1-q} $ represents the sum of any $ q $ eigenvalues of $ C $, we have the right inequality in \eqref{def:Dq}.
\end{remark}

As a consequence of Remark \ref{rem:cwzqDq}, {and the equivalence of $ Y(q) $ with \eqref{eq:equivyq}}, $ Y(q) $ holds if comparable weak $ Y(q) $ is satisfied when the Levi matrix has a nonvanishing eigenvalue.
Also, following a similar process, we can prove that if there is a unitary matrix $ A $ on an open set $ U $ such that $ A^*CA =\diag\lla{\mu_1,\cdots ,\mu_{n-1}} $ in non-decreasing order and $ A^*CA $ is varying smoothly in $ U $, then comparable weak $ Y(q) $ is fulfilled on $ U $ when $ D(q) $ is.

\begin{remark}\label{rem:eqDQcwyq} 	$ D(q) $ is equivalent to comparable weak $ Y(q) $ in the pseudoconvex case. In fact, in this case \eqref{def:Dq} becomes
	\begin{equation}\label{eq:Dqpseud}
	\ve \Tr[c_{jk}]\leq \sigma_q \leq (1-\ve) \Tr[c_{jk}].
	\end{equation}
	So if $ D(q) $ is satisfied, then both comparable weak $ Z(q) $ and comparable weak $ Z(n-1-q) $ are satisfied {with} $ \Upsilon=\ve Id $. Now assume comparable weak $ Z(q) $ is satisfied with $ \Upsilon$. Let $ \ve_1>0 $ be the smallest eigenvalue of $ \Upsilon $ and $ A $ be a unitary matrix such that $ A^*\Upsilon A =\diag\lla{\epsilon_1,\cdots,\epsilon_{n-1}}$. By pseudoconvexity, $ A^*CA $ is a positive semidefinite matrix and so
	\begin{align*}
	\sigma_q\geq\Tr(\Upsilon C)&= \sum_{j=1}^{n-1} \epsilon_j\Tr( I_j^j A^* CA )\geq \ve_1 \sum_{j=1}^{n-1}\Tr( I_j^j A^* CA ) = \ve_1 \Tr[C],
	\end{align*}
	where $ I_j^j $ is the matrix with $ 1 $ in $ j,j $ entry and zeroes everywhere else. 
	Likewise, if comparable weak $ Z(n-1-q) $ is satified there will exist $ 0<\ve_2<1 $ such that
	\begin{align*}
			Tr[C]-\sigma_{n-1-q} \leq (1-\ve_2) Tr[C].
	\end{align*}
	Selecting $ \ve=\min\lla{\ve_1,\ve_2} $ we obtain \eqref{eq:Dqpseud}.
\end{remark}

Thus, we can {view} our condition as an adaptation of $ D(q) $ for weak $ Y(q) $ manifolds.

\section{Subelliptic Estimates}

{In this section, we prove the following result about subelliptic estimates.}

\begin{thm}\label{thm:subell}
	Let $ M $ be a smooth, compact, orientable 
	CR manifold of hypersurface type. If the comparable weak $ Y(q) $ condition and the finite type property are satisfied in $ M $, then there exist constants $ C>0 $ and $ \ve>0 $ such that
	\begin{equation}\label{eq:subell}
	\norm{f}_{H^{\ve}_q(M)}^2 \leq C\pare{ \norm{\dbarb f}_{L^2_{q+1}(M)}^2+\norm{\dbarbs f}_{L^2_{q-1}(M)}^2 +\norm{f}_{L^2_q(M)}^2 }
	\end{equation}
	for every smooth $ (0,q) $-form $ f $.
\end{thm}

Observe that by \eqref{thm:horm} and \eqref{eq:horobs}, it is sufficient to control the $ L^2 $ norms of $ L_j f $ and $ \bar{L}_j f $ by the energy $ Q_b(f,f):=\norm{\dbarb f}_{L^2_{q+1}}^2+\norm{\dbarbs f}_{L^2_{q-1}}^2  $. We do this locally in Proposition \ref{prop:control1} using our comparability condition. To achieve the latter, we adapt an existing Morrey-Kohn type inequality \eqref{ec:mainid0} to one which allows us to use our comparability hypothesis, inequality \eqref{eq:mainid7}. We employ subtle integrations by parts and use microlocal tools to obtain \eqref{eq:mainid7}.

Since {the degree of the forms to which we apply the} norms $ \norm{\cdot}_{L^2} $, $ \norm{\cdot}_{L^2_q} $, $ \norm{\cdot}_{L^2_{q-1}} $, and $ \norm{\cdot}_{L^2_{q+1}} $, we will use just $ \norm{ \cdot } $ to denote them {in} this section. Also, we use $ \pare{\cdot,\cdot} $ to denote the inner product associated to $ \norm{ \cdot } $.

\subsubsection{Local setting and pseudodifferential operators}
Let $ \lla{\omega_j}_j $ be the dual forms of an orthonormal basis $ \lla{L_j}_{1\leq j\leq n-1} $ for $ T^{1,0}(M) $ in an open set $ U\subset M $, $ f=\sum_{J\in \I_q}f_J \bar{\omega}^J $ be a smooth $ (0,q) $-form with compact support in  $ U $,  and $ T $ be the vector used to define the Levi form.
We utilize the next notations
\begin{align*}
	{\norm{\bar L f}^2}&:= \sum_{J\in\I_q}\sum_{j=1}^{n-1}\norm{\lb_jf_J}^2, \qquad &
	{\norm{L f}^2}\ &:= \sum_{J\in\I_q}\sum_{j=1}^{n-1}\norm{\lba_j f_J }^2,\\
	\norm{\lb_\Upsilon f}^2&:=\sum_{J\in\I_q}\sum_{j,k=1}^{n-1}\pare{b^{\ba{k}j}\lb_kf_J,\lb_jf_J},
	\qquad&\norm{L_\Upsilon f}^2&:=\sum_{J\in\I_q}\sum_{j,k=1}^{n-1}\pare{\bkj \lba_j f_J,\lba_k f_J},
\end{align*}
where $ \Upsilon=i\sum_{j,k=1}^{n-1}b^{\bar{k}j}\lb_k\wg L_j $ is some real $ (1,1) $ vector defined on $ U $.
Locally we have $ \bar{L}^*_j=-L_j+a_j $ for some smooth functions $ a_j $, and
\[
K Q_b(f,f) \geq 
\norm{\bar L f}^2  +\sum_{I\in \I_q}\sum_{j,m=1}^{n-1}\pare{ [\bar{L}_m, \bar{L}^*_j]f_{jI},f_{mI}}-C\norm{f}^2 ,
\]
where $ K>1 $ and $ C>0 $ are constants just depending on $ U $ and $ a_j $'s functions.
If $ [L_j,\bar{L}_m] = c_{jm}T+ \sum_{l=1}^{n-1} \pare{ d_{jm}^l  L_l - \overline{d_{jm}^l} \bar{L}_l } $ for some smooth functions $ d_{jm}^l $ in $ U $, then we have
\begin{equation}
K Q_b(f,f) \geq \norm{\bar L f}^2 +
\re \sum_{I\in \I_{q-1}}\sum_{j,m=1}^{n-1}\pare{c_{jm}Tf_{jI},f_{mI}} + O((\norm{\bar L f}+ \norm{f})\norm{f}). \label{ec:mainid0}
\end{equation}

Denote by $ (\xi,\tau) $ the coordinates of dual space of $ \R^{2n-1}(x,t) $, and $ \sigma(T)=\tau $ (symbol of $ T $). Define the smooth functions
\[
\psi^+(\xi,\tau)=
\begin{cases}
0\ ,& \tau<0 \\
1\ ,& \tau\geq 1,
\end{cases}
\qquad
\psi^-(\xi,\tau)=
\begin{cases}
0\ ,& \tau>0 \\
1\ ,& \tau\leq -1,
\end{cases}
\qquad 
\psi^0=1-\psi^+-\psi^- .
\]
Let $ \theta $ be a compactly supported smooth function such that $ \theta \equiv 1 $ on $ U $.
Let $ \Psi^+ $, $ \Psi^- $ and $ \Psi^0 $ be the pseudodifferential operators of order zero with symbols $ \psi^+ $, $ \psi^- $ and $ \psi^0 $ respectively, and $ P^+ $, $ P^- $ and $ P^0 $ be the pseudodifferential operators of order zero defined by
\begin{align*}
P^+ f=\theta \Psi^+ f \qquad , \qquad  P^-f =\theta \Psi^- f \qquad , \qquad P^0 f=\theta \Psi^0 f .
\end{align*}

\subsubsection{Integrating by parts}
Let $ \Upsilon=i\sum_{j,k=1}^{n-1}b^{\bar{k}j}\lb_k\wg L_j $ and $ \tilde\Upsilon=i\sum_{j,k=1}^{n-1}B^{\bar{k}j}\lb_k\wg L_j $ be two real $ (1,1) $ vector fields defined on $ U $.
By definition of the operators above, we have
\begin{equation}\label{eq:decompose}
\norm{\lb_jf_J}^2=\norm{\lb_j P^+f_J}^2+\norm{\lb_j P^-f_J}^2+\norm{\lb_j P^0f_J}^2+\text{ ``cross terms'' }
\end{equation}
where the cross terms include terms of {the} type
\[
\pare{\lb_j P^+f_J,\lb_j P^-f_J},\ \pare{\lb_j P^0f_J,\lb_j P^-f_J}, \ \pare{\lb_j P^+f_J,\lb_j P^0f_J}.
\]

We can control {the latter by}
\begin{align*}
\pare{\lb_j P^+f_J,\lb_j P^0f_J}&
= \pare{(P^0)^*P^+ \lb_j f_J,\lb_jf_J}+O((\norm{\bar Lf}+\norm{\bar LP^+f})\norm{f})\\
& = O( (Q_b(f,f)+\norm{f}^2)^{\frac{1}{2}}\norm{\bar Lf}+(\norm{\bar Lf}+\norm{\bar LP^+f})\norm{f}+\norm{f}^2).
\end{align*}
In the last equality, we use the fact that, up to smooth terms, the symbol of $ (P^0)^*P^+ $ is {supported when} $ \abs{\tau}\leq 1 $, and the ellipticity of $ Q_b $ in complex directions. This ellipticity of $ Q_b $ also implies that $ \norm{\lb_j P^0f_J}^2 = O(Q_b(f,f)+\norm{f}^2)$. Because of the support of {the symbols of} $ P^- $ and $ P^+ $, the operator $ (P^-)^*P^+ $ is of order $ \leq -1 $. So we obtain
\begin{align*}
\pare{\lb_j P^+f_J,\lb_j P^-f_J}&
= \pare{(P^-)^*P^+ \lb_j f_J,\lb_jf_J}+O((\norm{\bar Lf}+\norm{\bar LP^-f})\norm{f})\\
& = O((\norm{\bar Lf}+\norm{\bar LP^-f})\norm{f}+\norm{f}^2).
\end{align*}

To avoid {an overload of error terms}, we will use $ E(f) $ to represent terms of type
\begin{align*}
O((\norm{\bar Lf}+\norm{\bar LP^+f}+\norm{\bar LP^-f})\norm{f}+\norm{f}^2),\
O((\norm{Lf}+\norm{LP^+f}+\norm{LP^-f})\norm{f}+\norm{f}^2),\\
O( (Q_b(f,f)+\norm{f}^2)^{\frac{1}{2}}\norm{\bar Lf}+Q_b(f,f)+\norm{f}^2). \qquad\qquad\qquad\qquad
\end{align*}
These terms are benign in the end. After summing in $ j $ and $ J $, from \eqref{eq:decompose}, we have
\begin{equation}\label{ec:mainid1} 
\norm{\bar Lf}^2=\norm{\bar LP^+f}^2+\norm{\bar LP^-f}^2+ E(f).
\end{equation}
In the same way, we have
\begin{equation}\label{ec:mainid11} 
\norm{Lf}^2=\norm{LP^+f}^2+\norm{LP^-f}^2+ E(f).
\end{equation}
{We introduce $ \Upsilon $ an then integrate by part ot obtain}
\begin{align}
\norm{\bar L P^+f}^2&
=\norm{\bar LP^+f}^2 - \norm{\lb_\Upsilon P^+f}^2 
 + \norm{L_\Upsilon P^+f }^2 \nn \\
&\quad  - \sum_{J\in \I_q}\sum_{j,m=1}^{n-1}\pare{b^{jm}c_{jm}TP^+f_{J}, P^+f_{J}} + E(f).
\label{ec:mainid2}
\end{align}
Similarly, performing {integration} by parts again, we have
\begin{align}
\norm{\bar L P^-f}^2&=\norm{LP^-f}^2 - \sum_{J\in \I_q}\sum_{j=1}^{n-1}\pare{c_{jj}TP^-f_{J}, P^-f_{J}} + E(f) \nn
\\
&=  \norm{LP^-f}^2 -  \norm{L_{\tilde\Upsilon} P^-f}^2 + \norm{ \lb_{\tilde\Upsilon} P^- f }^2 
 +  \sum_{J\in \I_q}\sum_{j,k=1}^{n-1}\pare{B^{kj}c_{jk}TP^-f_{J}, P^-f_{J}} \nn \\ 
&\quad - \sum_{J\in \I_q}\sum_{j=1}^{n-1}\pare{c_{jj}TP^-f_{J}, P^-f_{J}} + E(f). \label{ec:mainid4}
\end{align}

On the other hand
\begin{align*} 
\pare{c_{jm}Tf_{jI},f_{mI}} &=  \pare{c_{jm}TP^+f_{jI},P^+f_{mI}} + \pare{c_{jm}TP^-f_{jI},P^-f_{mI}} \\
&\quad+ \pare{c_{jm}TP^0f_{jI},P^0f_{mI}}+\text{cross terms},
\end{align*}
where {the} cross terms include
\begin{align*}
 &\pare{c_{jm}TP^+f_{jI},P^-f_{mI}} ,\ \pare{c_{jm}TP^+f_{jI},P^0f_{mI}} , \ \pare{c_{jm}TP^-f_{jI},P^+f_{mI}} , \ \pare{c_{jm}TP^+f_{jI},P^0f_{mI}}\\
 & \pare{c_{jm}TP^0f_{jI},P^+f_{mI}}, \ \pare{c_{jm}TP^0f_{jI},P^-f_{mI}}.
\end{align*}
Since $ P^+P^- $ is of order $ \leq -1 $, up to smooth terms, the symbols of operators $ TP^0 $, $ P^0P^+ $ and $ P^0P^- $ are supported {when} $ \abs{\tau}\leq 1 $, and from the ellipticity of $ Q_b $ {in the complex} directions, we have
\begin{align}\label{ec:mainid5} 
	\pare{c_{jm}Tf_{jI},f_{mI}} =  \pare{c_{jm}TP^+f_{jI},P^+f_{mI}} + \pare{c_{jm}TP^-f_{jI},P^-f_{mI}} + E(f).
\end{align}

Putting \eqref{ec:mainid1}, \eqref{ec:mainid2}, \eqref{ec:mainid4} and \eqref{ec:mainid5} into \eqref{ec:mainid0} we have
\begin{align}
KQ_b(f,f) &+ C\norm{f}^2 \geq \norm{ \bar L P^+f }^2 - \norm{ \lb_{\Upsilon} P^+f }^2 + \norm{ L_{\Upsilon} P^+f }^2 \nn \\
& +\norm{ L P^-f }^2 - \norm{ L_{\tilde\Upsilon} P^-f }^2 + \norm{ \lb_{\tilde\Upsilon} P^-f }^2 \nn \\
& + \sum_{I\in \I_{q-1}}\sum_{j,m=1}^{n-1} \pare{c_{jm}TP^+f_{jI},P^+f_{mI}} -  \sum_{J\in \I_{q}}\sum_{j,m=1}^{n-1} \pare{b^{mj}c_{jm}TP^+f_{J},P^+f_{J}} \nn \\
& + \sum_{I\in \I_{q-1}}\sum_{j,m=1}^{n-1} \pare{c_{jm}TP^-f_{jI},P^-f_{mI}} +  \sum_{J\in \I_{q}}\sum_{j,m=1}^{n-1}  \pare{B^{mj}c_{jm}TP^-f_{J},P^-f_{J}} \nn \\ 
& - \sum_{J\in \I_{q}}\sum_{j=1}^{n-1} \pare{c_{jj}TP^-f_{J},P^-f_{J}} + E(f).
\label{eq:mainid6}
\end{align}

Let $ H=[h_{jk}] $ be the matrix with entries
$
h_{jk} =  \delta_j^k \pare{ Tr[c_{rs}] - \sum_{r,s=1}^{n-1} B^{rs}c_{rs} } /{q} - c_{jk}.
$
Observe that if $ \lambda $ is an eigenvalue of $ [c_{jk}] $, then $ \pare{ Tr[c_{rs}] - \sum_{r,s=1}^{n-1} B^{rs}c_{rs}} /q - \lambda $ is an eigenvalue of $ H $. So, if the  sum of any $ n-q-1 $ eigenvalues of the matrix $ c_{jk} - \sum_{rs} B^{rs}c_{rs}/(n-q-1) $ is nonnegative, then the sum of any $ q $ eigenvalues of $ H $ is {nonnegative}, and vice versa. From \eqref{eq:mainid6}, we have our adapted Morrey-Kohn type inequality
\begin{align}
K&Q_b(f,f) + C\norm{f}^2 \geq \norm{ \bar L P^+f }^2 - \norm{ \lb_{\Upsilon} P^+f }^2 + \norm{ L_{\Upsilon} P^+f }^2 \nn \\
& +\norm{ L P^-f }^2 - \norm{ L_{\tilde\Upsilon} P^-f }^2 + \norm{ \lb_{\tilde\Upsilon} P^-f }^2 \nn \\
& + \sum_{I\in \I_{q-1}}\sum_{j,m=1}^{n-1} \pare{ \pare{ c_{jm} -\delta_j^k \tfrac{  \sum b^{rs}c_{sr}}{q} } TP^+f_{jI},P^+f_{mI}} \nn \\
& - \sum_{I\in \I_{q-1}}\sum_{j,m=1}^{n-1} \pare{h_{jm}TP^-f_{jI},P^-f_{mI}} +E(f).
 \label{eq:mainid7}
\end{align}

 \subsubsection{Using the comparable condition}

\begin{proposition}\label{prop:control1}
	Let $ M $ be a smooth CR manifold of hypersurface type.
	If the comparable weak $ Y(q) $ condition is satisfied at an open set $ U\subset M $, then
	\begin{equation}
	\norm{\bar L f}^2+\norm{ L f }^2 \lesssim Q_b(f,f)+\norm{f}^2, \qquad \label{eq:mainid8}
	\end{equation}
	for all smooth $ (0,q) $-form $ f $ supported on $ U $.
\end{proposition}

\begin{proof}
	Consider  $ \Upsilon $ coming from the comparable weak $ Z(q)$ condition and $ \tilde\Upsilon $ from the comparable weak $ Z(n-q-1) $ condition. By A) in Definition \ref{def:comp}, the matrices $ [\delta_j^k - b^{jk}] $ and $ [\delta_j^k - B^{jk}] $ have eigenvalues between $ 0 $ and $ 1 $. So there exists $ \delta>0 $ such that
\begin{align}
	\norm{\bar LP^+f}^2- \norm{\lb_\Upsilon P^+ f}^2&\geq 2\delta \norm{\bar LP^+ f}^2, \quad
	\norm{LP^-f}^2- \norm{L_{\tilde\Upsilon} P^- f}^2\geq 2\delta \norm{LP^- f}^2,  \label{eq:mainid9}\\
	\norm{ L_\Upsilon P^+ f }^2&\geq 2\delta \norm{L P^+ f}^2, \quad
	\norm{ \lb_{\tilde\Upsilon} P^- f }^2\geq 2\delta \norm{\bar L P^- f}^2 .\label{eq:mainid12}
\end{align}
From B) in Definition \ref{def:comp} and by sharp G\r{a}rding inequality we have for some constant $ C>0 $
\begin{align}
\re \lla{ \sum_{I\in \I_{q-1}}\sum_{j,m=1}^{n-1} \pare{ \pare{ c_{jm} -\delta_j^k \frac{  \sum b^{rs}c_{sr}}{q} } TP^+f_{jI},P^+f_{mI}} } & \geq - C \sum_{J\in \I_ q}\norm{ P^+ f_J }^2, \label{eq:mainid13} \\
\re\lla{ \sum_{I\in \I_{q-1}}\sum_{j,m=1}^{n-1} \pare{h_{jm}TP^-f_{jI},P^-f_{mI}} } & \geq - C \sum_{J\in \I_ q}\norm{ P^- f_J }^2. \label{eq:mainid14}
\end{align}

Taking \eqref{eq:mainid9}--\eqref{eq:mainid14}, \eqref{ec:mainid1} and \eqref{ec:mainid11} into \eqref{eq:mainid7} we obtain
\begin{align*}
KQ_b(f,f) + C\norm{f}^2 &\geq \delta\pare{\norm{\bar Lf}^2+\norm{Lf}^2} \\
&\quad +\delta\pare{ \norm{\bar LP^+ f}^2 + \norm{LP^- f}^2 +\norm{L P^+ f}^2+ \norm{\bar LP^-f}^2}+E(f).
\end{align*}
Using a small constant/large constant argument, we can absorb terms appearing in $ E(f) $ into the others that appear in this last inequality. This proves \eqref{eq:mainid8}.
\end{proof}

\begin{proof}[Proof of Theorem \ref{thm:subell}]
	 We get \eqref{eq:subell} locally by \eqref{thm:horm}, \eqref{eq:horobs}, and Proposition \ref{prop:control1} for some $ \ve>0 $ and $ C>0 $. We obtain the global estimate through a partition of unity argument on $ M $.
\end{proof}

When $ M $ is a weak $ Y(q) $ manifold, {it was observed in \cite{CoRa21}, that the points in $ M $ are divided into two sets, one where $ Y(q) $ is satisfied and one where} it is not. These last points are called by weak $ Y(q) $ points. 
Since we {proved} \eqref{eq:subell} locally and $ 1/2 $-subelliptic estimates hold where $ Y(q) $ is satisfied, we can localize the finite type and our comparable condition on neighborhoods of weak $ Y(q) $ points. We establish this in the next statement.

\begin{prop}\label{thm:subelllocalconditions}
	Let $ M $ be  a smooth, compact, orientable weak $ Y(q) $ manifold of hypersurface type. Let $ S $ be the set of weak $ Y(q) $ points. Suppose $ S $ is compact, the comparable weak-$ Y(q) $ condition and the finite type property are satisfied on a neighborhood of $ S $. There exist constants $ C>0 $ and $ \ve>0 $ such that
	\begin{equation}\label{eq:subelllocalconditions}
	\norm{f}_{H^{\ve}_q(M)}^2 \leq C\pare{ \norm{\dbarb f}_{L^2_{q+1}(M)}^2+\norm{\dbarbs f}_{L^2_{q-1}(M)}^2 +\norm{f}_{L^2_q(M)}^2 }
	\end{equation}
	for every smooth $ (0,q) $-form $ f $.
\end{prop}

Observe that, if we have closed range estimates at level of $ (0,q) $-forms, we will have \eqref{eq:subelllocalconditions} and \eqref{eq:subell} with $ \norm{H_q f}^2_{L^2_q(M)}$ instead of $ \norm{f}^2_{L^2_q(M)} $, where $ H_q $ represent the orthogonal projection on $ \Ker(\Box_b) $.

\subsection{Corollaries of subelliptic estimates}
Here we provide helpful Sobolev estimates to prove the maximal $ L^p $ estimates for the complex Green operator in the next section.

\begin{cor}\label{cor:2.2}
	Let $ U\subset M $,  $ \zeta,\zeta^\prime\in C^\infty_0(U) $ with $ \zeta \prec \zeta^\prime $, and $ s\geq 0 $. If the comparable weak-$ Y(q) $ and the finite type condition are satisfied at $ U $, then
	\begin{equation}\label{eq:2.3}
	\norm{\zeta u}^2_{H^{s+\ve}_{q}}\leq C_s	(\norm{\zeta^\prime \dbarb u}_{H^{s}_{q+1}}^2+\norm{\zeta^\prime \dbarbs u}^2_{H^{s}_{q-1}} +\norm{\zeta^\prime u}^2 )
	\end{equation}
	for all smooth $ (0,q) $-form $ u $.
\end{cor}

\begin{proof}
	Let $ n\in\N $, $ s=n\ve $ and $ \tilde{u}=\zeta u $. Then
	\begin{align}\label{eq:cor10}
	\norm{\tilde{u}}_{H^{s+\ve}_{q}}^2\lesssim \norm{\Lambda^s\tilde{u}}_{H^{\ve}_{q}}^2 \lesssim \norm{ \dbarb \Lambda^s \tilde{u} }^2 + \norm{ \dbarbs \Lambda^s \tilde{u} }^2+\norm{\Lambda^s \tilde{u}}.
	\end{align}
	We start {by} estimating $ \norm{ \dbarb \Lambda^s \tilde{u} }^2 $. 
	Observe that
	\begin{align*}
		\norm{\dbarb \Lambda^s \tilde{u}}^2 &= \pare{ \Lambda^s\dbarb \tilde{u},\dbarb \Lambda^s\tilde{u} }+\pare{ \corc{\dbarb,\Lambda^s}\tilde{u}, \dbarb\Lambda^s\tilde{u} }, \\
	\abs{ \pare{ \Lambda^s\dbarb \tilde{u},\dbarb \Lambda^s\tilde{u} } } &\lesssim l.c. \norm{\dbarb\tilde{u}}_{H^{s}_{q+1}}^2+ s.c.\norm{\dbarb \Lambda^s \tilde{u}}^2 ,\\
	\abs{ \pare{ \corc{\dbarb,\Lambda^s}\tilde{u}, \dbarb\Lambda^s\tilde{u} } } &\lesssim \norm{P^s \tilde{u}}\norm{\dbarb\Lambda^s\tilde{u}}\leq l.c.\norm{\Lambda^s \tilde{u}}^2+s.c.\norm{\dbarb\Lambda^s\tilde{u}}^2 ,
	\end{align*}
	where $ P^s $ denotes an operator of order $ s $, $ l.c $ and $ s.c. $ denote some large and small constant, respectively. Appropriate small constants allow us to absorb the term $ \norm{\dbarb\Lambda^s\tilde{u}}^2 $ and obtain
	\begin{align}\label{eq:cor11}
		\norm{\dbarb \Lambda^s \tilde{u}}^2 \lesssim \norm{\dbarb\tilde{u}}_{H^{s}_{q+1}}^2 + \norm{\Lambda^s \tilde{u}}^2.
	\end{align}
	In the same way
	\begin{align}\label{eq:cor12}
	\norm{\dbarbs \Lambda^s \tilde{u}}^2 \lesssim  \norm{\dbarbs\tilde{u}}_{H^{s}_{q-1}}^2 + \norm{\Lambda^s \tilde{u}}^2.
	\end{align}
	Using \eqref{eq:cor11} and \eqref{eq:cor12} into \eqref{eq:cor10}, we have
	\[
	\norm{\zeta u}^2_{H^{s+\ve}_{q}} \lesssim \norm{\zeta'\dbarb{u}}_{H^{s}_{q+1}}^2+ \norm{\zeta'\dbarbs{u}}_{H^{s}_{q-1}}^2 + \norm{\zeta'{u}}_{H^{s}_{q}}^2.
	\]
	By \eqref{eq:subell}, an induction process, and Sobolev embedding theorems, we have \eqref{eq:2.3} for every $ s=n\ve  $. We obtain \eqref{eq:2.3} for every $ s\geq 0 $ using interpolation.
\end{proof}

	\begin{remark}
	{The combination of 1) the hypotheses of Theorem \ref{thm:subell}, 2) a closed range estimate at the level of $ (0,q) $-forms, 3) the usual elliptic regularization process, and 4)}
	the estimates in Corollary \ref{cor:2.2} allows us to imply that $ K_q:C^\infty_{0,q}(M)\rightarrow C^\infty_{0,q}(M) $. 
	{The same property holds for Szeg\"o projection $ S_q=I-\dbarbs\dbarb K_q$, $ S_q^\prime =I-\dbarb\dbarbs K_q $, 
	and $ H_q=S_q+S_q^\prime -I $.}
	\end{remark}

\begin{cor}\label{cor:2.3}
	Let $ U\subset M $,  $ \zeta,\zeta^\prime \in C^\infty_0(U) $ with $ \zeta \prec \zeta^\prime $.  Then
	\begin{equation*}\label{eq:cor2.31}
	\norm{\zeta \phi}_{H^{s+\ve/2}_{q-1}}\leq C_s(\norm{\zeta^\prime \dbarb \phi}_{H^{s}_{q}}+\norm{\phi}), \qquad \forall \ \text{smooth} \ \phi\perp \ker(\dbarb^{ 0,q-1})
	\end{equation*}
	and
	\begin{equation*}\label{eq:cor2.32}
	\norm{\zeta \phi}_{H^{s+\ve/2}_{q+1}}\leq C_s(\norm{\zeta^\prime \dbarbs \phi}_{H^{s}_{q}}+\norm{\phi}), \qquad \forall \ \text{smooth} \ \phi\perp \ker(\dbarb^{* \ 0,q+1})
	\end{equation*}
	{Here $ C_s $ is just depending on $ s $, $ \zeta $ and $ \zeta^\prime $.}
\end{cor}
The proof of Corollary \ref{cor:2.3} follows the same lines as in Corollary 2.3 in \cite{koenig02}.

\section{Maximal $ L^p $ estimates}

In this section, in addition to comparable weak $ Y(q) $ and finite type conditions assumed in $ M $ as in the previous section, we also assume closed range estimates {at the level of $ (0,q) $-forms.} 
In particular, 
for each point in $ M $, there exists a neighborhood $ U $ where
we have the following two estimates
\begin{align}
	\norm{u}^2_{H^\ve_{q}} &\lesssim \sum_{K\in I_q}\sum_{j=1}^{2n-2}\norm{ X_{j}u_K }^2_{L^2} , \label{ineq:finitetype} \\
	\sum_{K\in I_q}\sum_{j=1}^{2n-2}\norm{ X_{j}u_K }^2_{L^2}&\lesssim\norm{\dbarb u}^2_{L^2_{q+1}}+\norm{\dbarbs u}^2_{L^2_{q-1}}+\norm{u}^2_{L^2} , \label{ineq:compweakyq}
\end{align}
for every smooth form $ u=\sum_{I\in \I_q} u_I \bar{\omega}^I$ with compact support in $ U $. Estimate \eqref{ineq:finitetype} comes from the finite type condition, and \eqref{ineq:compweakyq} from the comparable condition. 
The argument we follow appears in \cite{koenig02}. 
Unlike the latter, the closed range hypothesis at fixed form level $ q $ only allows us to guarantee the existence of $ K_q $ {(and not $ K_{q-1} $ or $ K_{q+1} $)}.
We endow $ M $ with a structure that turns it into a space of homogeneous type in the spirit of \cite{NagelSteinWainger}.

\subsection{Non isotropic metric}\label{sectionhomogeneoustype}
Let $ U $ be an open set in $ M $, $ \lla{L_j}_{j=1}^{n-1} $ be a basis of $ T^{1,0}(U) $, 
$X_j=(L_j+\bar{L}_j)/2$, $X_j=(L_j-\bar{L}_j)/(2i)$,
and $ T $ be vector used in the Levi form.
For each multi-index $ I= (i_1,...,i_k)\in \lla{1,...,2n-2}^k $ we define $ \lambda_I $ and $ \alpha_I^j $ on $ U $ such that
\[
[X_{i_k},[...,[X_{i_2},X_{i_1}]]]=-i\lambda_IT+\sum_{j=1}^{2n-2}\alpha_{I}^j X_j.
\]
For each $ x\in U $ and $ r>0 $, we make
\[
\Lambda_l(x)=\pare{ \sum_{2\leq |I|\leq l } |\lambda_{I}|^2 }^{1/2}
\quad \text{ and } \quad \Lambda(x,r)=\sum_{l=2}^{P}\Lambda_{l} (x) r^l .
\]
$ \Lambda^2_l $ is smooth, and $ \Lambda_P(x)\neq 0 $ in $ U $ because of the finite type condition. For $ x,y \in U$, the natural {non isotropic} distance between $ x$ and $ y $ corresponding to $ \lla{X_j} $ is given by
\begin{align*}
&\rho(x,y):=\inf \left\{ \delta >0 \ | \ \exists \phi:[0,1]\to U \text{ continuous and piecewise smooth }\right.
	\text{such that } \phi(0)=x ;\\
&\qquad \qquad  \phi(1)=y ;\ \text{ and } \phi^\prime(t) = \sum_{j=1}^{2n-2} \gamma_j(t)X_j \text{ a.e.}  \left. \text{ with } \vab{\gamma_j(t)}<\delta \text{ for each } 1\leq j\leq 2n-2 \right\} .
\end{align*}

Let denote $ B_0 $ the unit ball in $ \R^{2n-1} $ and define
\[
g_{x,r}(u)=\exp(ru_1X_1+...+ru_{2n-2}X_{2n-2}-\Lambda(x,r)u_{2n-1}iT )(x)
\]
for each $ x\in U $ and $ r $ sufficiently small such that $ g_{x,r} $ is a diffeomorphism on $ B_0 $. We write $ \tilde{B}(x,r) =g_{x,r}(B_0) $ and denote the Jacobian of $ g_{x,r} $ by $ J_{x,r} $. If $ f $ is a function on $ U $ we define  $ \hat{f}=f\circ g_{x,r} $ on $ B_0 $, and for a function $ \phi $ on $ B_0 $ we define $ \check{\phi}=\phi\circ g_{x,r}^{-1} $ on $ \tilde{B}(x,r) $. Also, we define the {scaled pullbacks} $ \hat{X}_j $ of vectors $ X_j $ by
$ \hat{X}_j=r g^*_{x,r}(X_j) $, where $ g^*_{x,r} $ is the formal pullback of $ g_{x,r}:B_0\rightarrow \tilde{B}(x,r) $. So, observe that 
\begin{equation}\label{eq:formscaling}
	 \pare{ X_j f }\hat{}=r^{-1}\hat{X}_j\hat{f}.
\end{equation}
With this configuration, the next result is established in \cite[Theorem 2.4]{koenig02}.

\begin{thm}\label{teo:2.4}
	Assume (only)  that $ U\subset M$  is of finite type. Then:
	\begin{enumerate}[(a)]
		\item $ B(x,c_1r)\subset \widetilde{B}(x,r)\subset B(x,c_2r)$ for constants $ c_1,c_2 $ independent of $ x $ and $ r $ (here the ball $ B(x,cr) $ is obtained throughout $ \rho $).
		\item $ |B(x,r)|\approx |\widetilde{B}(x,r)|\approx r^{2n-2}\Lambda(x,r) $ uniformly in $ x $ and $ r $.
		\item $ |J_{x,r}(u)|\approx r^{2n-2}\Lambda(x,r) $ uniformly in $ x $ and $ r $.
		\item $ |\frac{\p^\alpha}{\p u^\alpha}J_{x,r}(u)|\lesssim  r^{2n-2}\Lambda(x,r) $ uniformly in $ x $ and $ r $ for each multiindex $ \alpha $.
		\item The coefficients of the $ \hat{X}_j $ (expressed in $ u $ coordinates), together with their derivatives, are bounded above uniformly in $ x $ and $ r $.
		\item\label{itemf} The vector fields $ \hat{X}_1,\dots ,\hat{X}_{2n-2}$ are of finite type, and $ |\det(\hat{X}_1,\dots ,\hat{X}_{2n-2},Z)|\geq c_3 $ for a commutator $ Z $ (of the $ \hat{X}_j $) of length $ \leq P $ such that 
		$ \hat{X}_1,\dots ,\hat{X}_{2n-2},Z $ span the tangent space in $ B_0 $. Here $ c_3 $ is a positive constant that is independent of $ x $ and $ r $.
	\end{enumerate}
\end{thm}
	A restatement of \eqref{itemf} allows us to write $ \frac{\p}{\p u_i}= \sum_{j=1}^{2n-2}b_{ij}\hat{X}_j+b_{i,2n-1}Z $, $i=\ovl{1,n-1}$, so that the coefficients $ b_{ij} $, together with their derivatives, are bounded above uniformly in $ x $ and $ r $.\\

We will obtain the estimates for the complex Green and Szeg\"o projection kernels and more in terms of $ \rho(x,y) $  {by rescaling appropriate} uniform estimates from $ B_0 $.  For this, we will need estimates analogous to \eqref{eq:2.3} in $ B_0 $.
In $ B_0 $ we define $ \hat{L}_j:= \hat{X}_j+i\hat{X}_{j+n-1}  $, $ \hat{\bar{L}}_j:= \hat{X}_j-i\hat{X}_{j+n-1} $,
$ \dbarb^B $ the tangential Cauchy Riemann operator associated to the CR structure given by $ \{ \hat{L}_1,\cdots,\hat{L}_{n-1} \} $, and $ (\dbarb^B)^* $ the adjoint of $ \dbarb^B $.
Let us denote $ \hat{\bar{\omega}}_1,\cdots,\hat{\bar{\omega}}_{n-1} $ the  $ (0,1) $-forms dual to $ \hat{\bar{L}}_1,\cdots,\hat{\bar{L}}_{n-1} $, respectively. 
For each $ (0,q) $-form $ \phi=\sum_{J\in \I_q}\phi_J\bar{\omega}^J $ in $ U $ we set $ \hat{\phi}=\sum_{J\in \I_q}\hat{\phi}_J\hat{\bar{\omega}}^J $ and $ \norm{\hat{\phi}}^2=\sum_{J\in \I_q}\norm{\hat{\phi}_J}^2 $.
Also, we have for every function $ \phi $ in $ U $ 
\[
\qquad \qquad  \qquad \pare{\bar{L}_j\phi}{\hat{}} =r^{-1}{ \hat{\bar{L}}_j }\hat{\phi}, \qquad \qquad \text{for all } 1\leq j\leq n-1.
\]
If $ a_j\in C^\infty(U) $ is such that $ \bar{L}^*_j=-L_j+a_j $,
we define $ \pare{\bar{L}^*_j}\hat{}:=-\hat{L}_j+r\hat{a}_j $. So we have
\[
\pare{\bar{L}^*_j \phi}\hat{} = r^{-1} \pare{\bar{L}^*_j}\hat{} \ \hat{\phi},
\text{ and }
\pare{\hat{\bar L}_j}^* = \pare{\bar{L}^*_j}\hat{}-J_{x,r}\pare{\hat{L}_j(\frac{1}{J_{x,r}\circ g^{-1}_{x,r
	}})\hat{}\ } = \pare{\bar{L}^*_j}\hat{}-J_{x,r}{\hat{L}_j(\frac{1}{J_{x,r}}) },
\]
where $ J_{x,r}\hat{L}_j(\frac{1}{J_{x,r}}) $ and all its derivatives are uniformly bounded by above in $ x $ and $ r $.

We define the operators $ \hat{\dbar}_b $ and $ \hat{\p}_b $ acting on $ (0,q) $-forms in $ B_0 $ to be the scaled pullback of $ \dbarb $ and $ \dbarbs $ given by
\begin{equation*}
(\dbarb \phi)\hat{} = r^{-1}\hat{\dbar}_b \hat{\phi}\ ,\qquad \qquad  (\dbarbs \phi)\hat{} = r^{-1}\hat{\partial}_b\hat{\phi}. 
\end{equation*}
Thus, from {the} definition of $ \dbarb $ and $ \dbarbs $ 
we will have $\hat{\dbar}_b=\dbarb^B $ and $ \hat{\partial}_b=(\dbarb^B)^*+ P^0 $
for some operator $ P^0 $ of order zero, uniformly in $ x $ and $ r $.

\subsection{Pulling back estimates on $ B_0 $} In this subsection, we pull back the estimates for $ \dbarb $ and $ \dbarbs $ in $ U $ to those in $ B_0 $ regarding $ \hat{\dbar}_b $ and $ \hat{\p}_b $.
The following result is one of the main estimates we use to scale $ L^2 $ estimates. It is established in \cite[Theorem 2.7]{koenig02}, and there is a version when $ n=2 $ in \cite[Lemma 11.4]{christ88}.

\begin{prop}\label{prop:2.7}
	Assume that $ U^\prime\subset M $ is of finite type. If $ B(x,r)\subset U\subset\!\subset U^\prime $, then
	\[
	\norm{u}_{L^2_q(B(x,r))}\leq Cr\pare{ \sum_{K\in \I_q}\sum_{j=1}^{2n-2}\norm{X_j u_K}_{L^2(U^\prime)} + \norm{u}_{L^2_q(U^\prime)} }
	\]
	for every smooth $ (0,q) $-form $ u=\sum_{K\in \I_q}u_K \bar{\omega}_K $ with $ 0\leq q\leq n-1 $. (Here $ C $ is a positive constant just depending on $ U $ and $ U^\prime $).
\end{prop}

\begin{remark}\label{rem:pullback1}
	From Theorem \ref{teo:2.4}, if $ f  $ is defined on $ \tilde{B}(x,r) $, we have
		\[
		\norm{f}_{L^2(\tilde{B}(x,r))}^2
		\lesssim \abs{B(x,r)} \norm{\hat{f}}_{L^2(B_0)}^2 
		\quad \text{ and } \quad 
		\norm{\hat{f}}_{L^2(B_0)}^2
		\lesssim \abs{B(x,r)}^{-1} \norm{f}_{L^2(\tilde{B}(x,r))}^2 ,
		\]
where the constants in the inequalities are independent of $ x $ and $ r $.
\end{remark}

\begin{prop}\label{prop:2.5}
	Fix $ \zeta,\zeta^\prime\in C^\infty_0(B_0) $ with $ \zeta\prec \zeta^\prime $. For smooth $ (0,q) $-forms $ v=\sum_{K\in \I_q}v_K\hat{\bar\omega}^K $ on $ B_0 $ and $ s\geq 0 $
	\begin{equation}\label{eq:251}
	\norm{\zeta v}_{H^{s+\ve}_{q}(B_0)}^2\leq C_s (\norm{\zeta^\prime \hat{\dbar}_bv }_{H^{s}_{q+1}(B_0)}^2 + \norm{\zeta^\prime \hat{\p}_bv }_{H^{s}_{q-1}(B_0)}^2 + \norm{v}_{{L^2_q(B_0)}}^2 ).
	\end{equation}
\end{prop}

\begin{proof}
	By \eqref{ineq:compweakyq} and Remark \ref{rem:pullback1}
	\begin{align*}
	\norm{\hat{X}_j(\zeta v)}^2_{L^2_q(B_0)}&
	\lesssim r^2\abs{B(x,r)}^{-1}\pare{\norm{ \dbarb(\zeta v)\check{}\  }^2_{\tilde{B}(x,r)} + \norm{ \dbarbs(\zeta v)\check{}\  }^2_{\tilde{B}(x,r)} + \norm{ (\zeta v)\check{}\  }^2_{\tilde{B}(x,r)}},\\
	&\lesssim 
	r^2\pare{\norm{ r^{-1}\hat{\dbar}_b(\zeta v) }^2_{L^2_{q+1}(B_0)} + \norm{ r^{-1}\hat{\p}_b(\zeta v)}^2_{L^2_{q-1}(B_0)} + \norm{ \zeta v  }^2_{L^2_q(B_0)}}.
	\end{align*}
	Then
	\begin{align*}
	\sum_{j=1}^{2n-2}\norm{\hat{X}_j(\zeta v)}^2_{L^2_q(B_0)}\leq C\pare{\norm{ \hat{\dbar}_b(\zeta v) }^2_{L^2_{q+1}(B_0)} + \norm{ \hat{\p}_b(\zeta v)}^2_{L^2_{q-1}(B_0)} + r^2\norm{ \zeta v  }^2_{L^2_q(B_0)}}.
	\end{align*}
	On the other hand, by $ (e) $ and $ (f) $ in Theorem \ref{teo:2.4} and \eqref{thm:horm}, we obtain
	\[
	\norm{\zeta v}_{H^\ve_{q}(B_0)}^2\lesssim \sum_{j=1}^{2n-2}\norm{\hat{X}_j(\zeta v)}^2_{L^2_q(B_0)} + \norm{\zeta v}^2_{L^2_q(B_0)},
	\]
	uniformly in $ x $ and $ r $. This proves \eqref{eq:251} for $ s=0 $. We obtain \eqref{eq:251} for $ s>0 $ using a similar argument to prove the Corollary \ref{cor:2.2}.
\end{proof}

	\begin{cor}
	\label{prop:2.6}
	Fix $ \zeta,\zeta'\in C^\infty_0(B_0) $ with $ \zeta\prec \zeta' $. In the following $ s\geq 0 $.
	\begin{enumerate}[a)]
	\item Let $ v $ and $ g $ be $ (0,q-1) $-forms on $ B_0 $, and suppose that $ \hat{\p}_b\hat{\dbar}_b v=g $. Then
	\begin{equation}\label{eq:k2.8}
	\norm{\zeta\hat{\dbar}_bv}_{H^{s+\ve}_{q}(B_0)}\leq C_s \pare{ \norm{\zeta' g}_{H^s_{q-1}(B_0)} +\norm{\hat{\dbar}_bv}_{L^2_q(B_0)}}.
	\end{equation}
	
	\item Let $ v $ and $ g $ be $ (0,q) $-forms on $ B_0 $, and suppose that $ \hat{\p}_b\hat{\dbar}_b v=g $ with $ \hat{\p}_b v=0 $. Then
	\begin{equation}\label{eq:k2.9}
		\norm{\zeta\hat{\dbar}_bv}_{H^{s+\ve/2}_{q+1}(B_0)}\leq C_s \pare{ \norm{\zeta' g}_{H^s_{q}(B_0)} +\norm{\hat{\dbar}_bv}_{L^2_{q+1}(B_0)}+\norm{v}_{L^2_q(B_0)}}.
	\end{equation}
	
	\item Let $ v $ and $ g $ be $ (0,q+1) $-forms on $ B_0 $, and suppose that $ \hat{\p}_b\hat{\dbar}_b v=g $. Then
	\begin{equation}\label{eq:k2.10}
		\norm{\zeta\hat{\p}_bv}_{H^{s+\ve}_{q}(B_0)}\leq C_s \pare{ \norm{\zeta' g}_{H^s_{q-1}(B_0)} +\norm{\hat{\p}_bv}_{L^2_q(B_0)}}.
	\end{equation}
	
	\item Let $ v $ and $ g $ be $ (0,q) $-forms on $ B_0 $, and suppose that $ \hat{\dbar}_b\hat{\p}_b v=g $ with $ \hat{\dbar}_b v=0 $. Then
	\begin{equation}\label{eq:k2.11}
		\norm{\zeta\hat{\p}_bv}_{H^{s+\ve/2}_{q-1}(B_0)}\leq C_s \pare{ \norm{\zeta' g}_{H^s_{q}(B_0)} +\norm{\hat{\p}_bv}_{L^2_{q-1}(B_0)}+\norm{v}_{L^2_q(B_0)}}.
	\end{equation}
	
	\item Let $ v $ and $ g $ be $ (0,q) $-forms on $ B_0 $, and suppose that $ \hat{\dbar}_b\hat{\p}_b v=g $ with $ \hat{\dbar}_b v=0 $. Then
	\begin{equation}\label{eq:k2.12}
		\norm{\zeta v}_{H^{s+3\ve/2}_{q}(B_0)}\leq C_s \pare{ \norm{\zeta' g}_{H^s_{q}(B_0)} +\norm{\hat{\p}_bv}_{L^2_{q-1}(B_0)}+\norm{v}_{L^2_q(B_0)}}.
	\end{equation}

	\item Let $ v $ and $ g $ be $ (0,q) $-forms on $ B_0 $, and suppose that $ \hat{\p}_b\hat{\dbar}_b v=g $ with $ \hat{\p}_b v=0 $. Then
	\begin{equation}\label{eq:k2.13}
		\norm{\zeta v}_{H^{s+3\ve/2}_{q}(B_0)}\leq C_s \pare{ \norm{\zeta' g}_{H^s_{q}(B_0)} +\norm{\hat{\dbar}_bv}_{L^2_{q+1}(B_0)}+\norm{v}_{L^2_q(B_0)}}.
	\end{equation}
	\end{enumerate}
	\end{cor}

	\begin{proof}
		Estimates \eqref{eq:k2.8} and \eqref{eq:k2.10} follow from Proposition \ref{prop:2.5}. \eqref{eq:k2.9} and \eqref{eq:k2.11} follow similarly {using ideas from the} to the proof of Corollary \ref{cor:2.3} . For \eqref{eq:k2.12} we use Proposition \ref{prop:2.5} and \eqref{eq:k2.11} as follows
		\begin{align*}
		\norm{\zeta v}_{H^{s+3\ve/2}_{q}(B_0)}&\lesssim \pare{ 
			\norm{\zeta'' \hat{\p}_b v}_{H^{s+\ve/2}_{q-1}(B_0)} + \norm{v}_{L^2_q(B_0)} }\\
		&\lesssim \pare{ 
			\norm{\zeta' g}_{H^{s}_{q}(B_0)}+ \norm{\hat{\p}_b v}_{L^2_{q-1}(B_0)} + \norm{v}_{L^2_q(B_0)} }.
		\end{align*}
		for smooth functions $ \zeta\prec\zeta''\prec\zeta' $ in $ C^\infty_0(B_0) $. \eqref{eq:k2.13} is proved in a similar way.
	\end{proof}
 
We {turn now to studying} the regularity properties of Szeg\"o Projection $ S_q $ and the Complex Green operator $ K_q $ throughout the following operators they define. 
For a $ (0,q) $-form $ f=\sum_{J\in \I_q} f_J \bar{\omega}^J $ and $ I\in \I_q $, let $ \la{K_qf, \bar{\omega}^I}\ra $ denote 
the function satisfying
$
K_q f=\sum_{I\in \I_q}\la{ K_q f,\bar{\omega}^I }\ra \bar{\omega}^I.
$
For every $ I,J\in \I_q $ we define the operators $ K_{q}^{IJ} $ by
\[
K_q^{IJ}(g):= \la{K_q(g \bar{\omega}^J),\bar{\omega}^I}\ra
\]
for any function $ g $ on $ M $. Since $ K_q $ is linear and bounded in $ L^2_q $, the operators $ K_q^{IJ} $ are linear and bounded in $ L^2(M) $. 
{More generally}, if $ F:L^2_q(M)\to L^2_q(M)$, we define $ F^{IJ} $ similarly. For example, if $ F=Id $, we have $ F^{IJ}(g)=g\delta_J^I $, where $ \delta^I_J $ is Kronecker's delta.

\subsection{Regularity properties of Szeg\"o Projection}
{The techniques for estimating $ K^{IJ}_q $ and $ S^{IJ}_q $ are similar, and they are the focus of this section.}
Let $ S_q^{IJ}(x,y) $ denote the distribution-kernel of $ S_q^{IJ} $. We use $ \Delta_U =\lla{(x,x): x\in U} $ to denote the diagonal of $ U\times U $.

First, we show that $ S_q^{IJ}(x,y) $ 
is smooth away from the diagonal. For this, it will suffice to show that the distribution-kernel of 
$ (I-S_q)^{IJ} $ is smooth {off of} the diagonal.

Let $ f $ be $ (0,q) $-form with coefficients in $ C^1_0(U_1) $ for some $ U_1\subset U$, and $ U_2\subset U $ be a neighborhood of $ \ovl{U_1} $. Then $ \dbarb(I-S_q)f=\dbarb f=0 $ in $ U_1^\complement $. Let $ \zeta , \zeta^\prime \in C^\infty_0(U\setminus \ovl{U}_2)$ such that $ \zeta \prec \zeta^\prime $. Then by Corollary \ref{cor:2.2}, we have
\begin{align}
\norm{\zeta(I-S_q) f}_{H^{s+\ve}_{q}(U)} &\lesssim C_s \pare{ \norm{\zeta^\prime \dbarb(I-S_q)f}_{H^s_{q}(U)}+ \norm{ \zeta^\prime \dbarbs(I-S_q)f }_{H^s_{q}(U)}+\norm{(I-S_q)f}_{L^2_q} }\nn \\
&\lesssim C_s\norm{f}_{L^2_q}{.} \label{eq:teo4.2}
\end{align}
Let $ V $ and $ V_2 $ be open sets such that $ \ovl{U}_2\subset V_2 \subset\!\subset V\subset\!\subset U $ and $ \zeta \equiv 1 $ in $ V\setminus \ovl{V}_2 $. By \eqref{eq:teo4.2}
\begin{align*}
\norm{(I-S_q)f}_{H^{s+\ve}_{q}(V\setminus \ovl{V}_2)}\leq \norm{\zeta(I-S_q) f}_{H^{s+\ve}_{q}(U\setminus \ovl{U}_2)}\leq \norm{\zeta(I-S_q) f}_{H^{s+\ve}_{q}(U)} \leq C_s\norm{f}_{L^2_q(U_1)},
\end{align*}
for every $ (0,q) $-form $ f $ with coefficients in $ C^1_0(U_1) ${.}
Then, for any function $ g\in C^1_0(U_1) $ and $ I,J\in \I_q $,
we obtain
\begin{align*}
\norm{(I-S_q)^{IJ}g}_{H^{s+\ve}(V\setminus \ovl{V}_2)}&
\leq \norm{ (I-S_q)(g\bar{\omega}^J) }_{H^{s+\ve}_q(V\setminus \ovl{V}_2)}
\lesssim C^\prime_s \norm{g \bar{\omega}^J}_{L^2_q(U_1)}\leq C_s^\prime \norm{ g }_{L^2(U_1)}.
\end{align*}
By {the Sobolev embedding theorem}, $ (I-S_q)^{IJ} : L^2(U_1)\rightarrow C^\infty(V\setminus \ovl{V_2}) $ is bounded for all $ I,J\in \I_q $. In particular
\begin{equation*}
\abs{ \int_{U_1} \frac{\p}{\p x^\alpha}(I-S_q)^{IJ}(x,y)g(y)dy } =
\Norm{\frac{\p}{\p x^\alpha}[(I-S_q)^{IJ}g]}_{L^\infty(V\setminus \ovl{V}_2)}\leq C_\alpha \norm{g}_{L^2(U_1)}.
\end{equation*}
By $ L^2 $ duality, we obtain
\[
\Norm{\frac{\p}{\p x^\alpha}(I-S_q)^{IJ}(x,\cdot )}_{L^2(U_1)}\leq C_{1,\alpha},\qquad\qquad \forall \ x\in V\setminus\ovl{V}_2.
\]
Since $ (I-S_q) $ is self-adjoint, $ (I-S_q)^{IJ}(z,w)=\ovl{(I-S_q)}^{JI}(w,z) $. In a similar way (now taking $ g\in C^1_0(V\setminus \ovl{V}_2) $), we also obtain
\begin{equation*}
\Norm{\frac{\p}{\p y^\beta}(I-S_q)^{IJ}(\cdot,y)}_{L^2(V\setminus \ovl{V}_2)}\leq C_{2,\beta},\qquad\qquad \forall \ y\in U_1.
\end{equation*}
Let $ \phi \in C^\infty_0(V\setminus \ovl{V}_2)$ and $ \psi\in C^\infty_0(U_1) $. From the last two inequalities, for any multi indexes $ \alpha $ and $ \beta $ there exist $ C_\alpha , C_\beta>0 $ such that
\begin{equation*}
\nnorm{ \frac{\p}{\p x^\alpha}\corc{ \phi(x)\psi(y) (I-S_q)^{IJ}(x,y)} }_{L^2(V,dy)}\leq C_\alpha, \qquad\qquad \forall x\in V\setminus \ovl{V}_2,
\end{equation*}
and
\begin{equation*}
\nnorm{ \frac{\p}{\p y^\beta}\corc{ \phi(x)\psi(y) (I-S_q)^{IJ}(x,y)} }_{L^2(V,dx)}\leq C_\beta,\qquad\qquad \forall y\in U_1.
\end{equation*}
A {standard} Fourier inversion argument (see \cite[Lemma 6.2]{RothStein76}) yields
\[
\nnorm{\frac{\p^\alpha}{\p x^\beta}\frac{\p^\alpha}{\p y^\beta}\corc{ \phi(x)\psi(y)(I-S_q)^{IJ}(x,y) } }_{L^2(V\times V,dxdy)}<\infty .
\]
By {the Sobolev embedding theorem}, we have $ \phi(x)\psi(y)(I-S_q)^{IJ}(x,y) \in C^\infty(V\times V)$. In particular, $ (I-S_q)^{IJ}(x,y) $ is smooth for $ (x,y)\in (V\setminus\ovl{V}_2)\times U_1 $, with $ U_1$ and $ (V\setminus \ovl{V}_2) $ disjoint. This proves  that $ (I-S_q)^{IJ}(x,y) $, and hence $ S_q^{IJ} $, is smooth away from diagonal $ \Delta_U $. 

	Similarly, we can prove that the distribution-kernel of $ S_q^{\prime IJ} $ is smooth away from diagonal $ \Delta_U $. 
	Also, by {a similar argument and Corollary} \ref{cor:2.3}, we can show that the distribution-kernels of $ (S_qK_q)^{IJ} $ and $ (S_q'K_q)^{IJ} $ are in $ C^\infty(U\times U \setminus \Delta_U) $ (note that $ S_qK_q =K_qS_q$ and $ S_q^\prime K_q =K_qS_q^\prime $). Since $ K_q=S_qK_q+S_q'K_q $, then $ K_q^{IJ}(x,y) $ is also in $ C^\infty(U\times U \setminus \Delta_U) $. Analogously it can be proved that the distribution-kernel of $ H_q^{IJ} $ is smooth in $ U\times U $.

	We say $ \psi $ is a {bump function} if it is a smooth function supported in $ B(x,r) $ and normalized
	{in the sense that there exists $ N>0 $ so that}
	$ \norm{D^l\phi}_{L^\infty} \leq r^{-l}$ for every $ 0\leq l\leq N $ and $ D^l = X_{j_1}X_{j_2}\cdots X_{j_l}$. Note that if $ \phi=\sum_{J\in \I_q}\phi_J\bar{\omega}^J $ such that $ \phi_J $ is a bump function for every $ J\in \I_q $, then we have
	\[
	\norm{\phi}^2_{L^2_q}\lesssim \abs{B(x,r)},\quad \norm{ \widehat{D}^k_z \hat{\phi}(z) }_{L^\infty_{q}(B_0)} \lesssim 1,\qquad \Norm{\frac{\p^\alpha}{\p z^\alpha}\hat{\phi}(z) }_{L^\infty_{q}(B_0)}\leq C_M,
	\]
	for $ \abs{\alpha}\leq M=M(N) $. For the last inequality, we used \eqref{itemf} of Theorem \ref{teo:2.4}, for which we assume  $ N $ is sufficiently large.

Now, we prove cancellation properties for the operator $ S^{IJ}_q $ and $ S^{\prime IJ}_q $ in terms of its action on bump functions, namely \eqref{eq:ebfS5}. 

Let $ \phi $ be a $ (0,q) $-form with coefficients {that are} bump functions supported in $ B(x,r) $. Let $ V\subset \widetilde{V}\subset B_0 $, $ \zeta,\zeta^\prime\in C^\infty_0(B_0) $ such that $ \zeta\prec\zeta^\prime $ and $ \zeta\equiv 1 $ in $ V $. In the following, we will scale with respect to $ g_{x,R} $ for some $ R>0 $, such that $ B(x,r)\subset\tilde{B}(x,R) $, chosen appropriately. Then by Proposition \ref{prop:2.5}, we have
\begin{align}
	\norm{\zeta ((I-S_q)\phi)\hat{}\  }^2_{H^{s+\ve}_{q}(B_0)} 
	\leq C_s\pare{ 
		\norm{R\zeta^\prime (\dbarb\phi)\hat{}\ }^2_{H^s_{q+1}(B_0)}+ \norm{ \pare{ (I-S_q)\phi }\hat{}\ }^2_{L^2_q(B_0)} }. \label{eq:mcSbf1}
\end{align}
For $ N $ sufficiently large, we obtain
\begin{align*}
	\norm{R\zeta^\prime (\dbarb\phi)\hat{}\ }_{H^s_{q+1}(B_0)}\leq C_{\zeta^\prime} \sum_{\abs{\alpha}\leq s+1} \Norm{ \frac{\p^\alpha \hat{\phi}}{\p z^\alpha} }_{L^2_q} \leq C_s.
\end{align*}
Also, since $ I-S_q $ is bounded in $ L^2_q(M) $ {by the closed range estimates and} Theorem \ref{teo:2.4}, we have
\begin{align*}
	\norm{ \pare{ (I-S_q)\phi }\hat{}\ }^2_{L^2_q(B_0)}&\lesssim \abs{B(x,R)}^{-1}\norm{(I-S_q)\phi}^2_{L^2_q(\widetilde{B}(x,R))} \lesssim \abs{B(x,R)}^{-1}\abs{B(x,r)} \lesssim 1.
\end{align*}
So, by \eqref{eq:mcSbf1}
\begin{align}\label{eq:SqSprimeq}
\norm{((I-S_q)\phi)\hat{}\ }_{H^{s+\ve}_{q}(V)}\leq \norm{ \zeta((I-S_q)\phi)\hat{}\ }_{H^{s+\ve}_{q}(B_0)}\leq C_s.
\end{align}
Using {the Sobolev embedding theorem}, for $ D^k=X_{j_1}X_{j_2}\cdots X_{j_k} $, we have
\[
\abs{ D^k(I-S_q)\phi(z) }\leq \norm{ D^k(I-S_q)\phi }_{L^\infty_{q}(g_{x,R}(V))}= R^{-k} \norm{ \widehat{D^k}((I-S_q)\phi)\hat{}\ }_{L^\infty_{q}(V)}\lesssim R^{-k}
\]
for all $ z $ such that $ z\in g_{x,R}(V) $. {Choosing $ R=c_1^{-1}r $ and $ V $ such that $ B(x,r)\subset g_{x,R}(V) $. Then by Theorem \ref{teo:2.4}, we} have $ B(x,r)\subset \tilde{B}(x,R)$ and also
\[
\abs{D^{k}(I-S_q)\phi(z)}\leq c_k r^{-k},\qquad\qquad z\in B(x,r).
\]

\begin{remark}\label{rem:ebfS6}
	The process above {can also} be done for $ S^\prime_q $. {In particular, the following estimates hold:}
	\begin{equation}\label{eq:ebfS5}
	\abs{D^k(S_q^{IJ}\psi)(z)}\leq C_k r^{-k}\ ,\qquad 	\abs{D^k(S_q^{\prime IJ}\psi)(z)}\leq C_k r^{-k} ,\qquad \forall z\in B(x,r)
	\end{equation}
	for a bump function $ \psi $, provided $ N=N(k) $ is chosen sufficiently large. In fact, following an analogous process above, we can prove \eqref{eq:SqSprimeq} for $ S_q^\prime $ instead of $ S_q $ for $ N $ sufficiently large.
\end{remark}

\subsection{Regularity properties of the Complex green operator}
This section aims to prove cancellation properties to $ K^{IJ}_q $, Proposition \ref{prop:ebfK}, and the next pointwise differential estimates for its distributional kernel. 

\begin{thm}\label{teo:5.1}
	Let $ M $ be a smooth, compact, orientable CR manifold of hypersurface type with closed range for $ \dbarb^{0,q-1} $ and $ \dbarb^{0,q} $. Assume {the comparable  weak $ Y(q) $ and finite type conditions} are satisfied in $ M $. 
	Then for each point in $ M $, there exists a neighborhood $ U $ such that $ K^{IJ}_q(x,y) \in C^\infty(U\times U \setminus \Delta_U) $.
	Moreover, for $ x\neq y $
	\begin{equation}\label{pointwisegreen}
		\abs{D_x^kD_y^lK_q^{IJ}(x,y)}\leq C_{k,l}\  \rho(x,y)^{2-k-l}\abs{B(x,\rho(x,y))}^{-1}
	\end{equation} 
	where $ D_x^k $ represents any composition $ X_{i_1}X_{i_2}\cdots X_{i_k} $ acting in the first variable $ x $ with $ (i_1,\cdots,\break i_k) \in \lla{ 1,2,\cdots,2n-2 }^k $ and similarly for $ D_y^l $ in the second variable $ y $. In fact, \eqref{pointwisegreen} is satisfied for operators $ (S_qK_q)^{IJ} $ and $ (S_qK_q^\prime)^{IJ} $.
\end{thm}

We use a process similar to the previously employed for $ S^{IJ}_q $ to obtain estimates \eqref{pointwisegreen} and those of cancellation type. However, operator $ K_q^{IJ} $ requires an extra decay of order two, which we prove in Proposition \ref{prop:ebfK} using the results of Proposition \ref{prop:2.7} and Corollary \ref{prop:2.6}{.
The canonical solutions for the $ \dbarb $ and $ \dbarbs $-equations are 
$ G_q=\dbarbs K_q $ and $ G_q^\prime=\dbarb K_q $,
respectively.} 
The next two propositions will be used to show Proposition \ref{prop:ebfK}.

\begin{proposition}\label{prop:5.3}
	Fix $ V\subset\!\subset U $ and assume that $ B(x,r) \subset V$. 
	
	\begin{enumerate}[(a)]
		\item\label{prop5.3a} If $ \dbarb^{*} u=f$ in $ M $ with $ u\in C^\infty_{0,q}(M) $, $ f\in C^\infty_{0,q-1}(M) $ and $ u\perp \ker{ \dbarb^{*} } $, then
		\begin{equation}\label{ineq3.2}
		\norm{u}_{L_{q}^2(B(x,r))}\leq C r \norm{f}_{L^2_{q-1}},
		\end{equation}
		where $ C $ is a positive constant independent of $ x $ and $ r $ (but it can depend on $ V $ and $ U $).\\
		
		\item\label{prop5.3b} If $ \dbarb u=f$ in $ M $ with $ u\in C^\infty_{0,q}(M) $, $ f\in C^\infty_{0,q+1}(M) $ and $ u\perp \ker{ \dbarb } $, then
		\begin{equation}\label{ineq3.2p}
		\norm{u}_{L_{q}^2(B(x,r))}\leq C r \norm{f}_{L^2_{q+1}}
		\end{equation}
		where $ C $ is a positive constant independent of $ x $ and $ r $ (but it can depend on $ V $ and $ U $).
	\end{enumerate}
\end{proposition}

\begin{proof}
	Note that $ \norm{\dbarb u}=
	0 $. Let $ V\subset\tilde V\subset U $, $ \zeta\in C^\infty_{0}(U) $ such that $ \zeta =1 $ on $ \tilde{V} $. 
	Applying \eqref{ineq:compweakyq} {to} $ \zeta u $, we obtain
	\begin{align}
	\sum_{J\in \I_q}\sum_{k=1}^{2n-2} \norm{X_k u_J}^2_{L^2(\tilde{V})}&\leq
	\sum_{J\in \I_q}\sum_{k=1}^{2n-2} \norm{X_k (\zeta u_J)}^2_{L^2(U)} 
	\leq 
	C_\zeta\pare{ \norm{\dbarbs u}^2_{L^2_{q-1}(U)}+\norm{u}^2_{L^2_q(U)} } \nn  \\
	&\lesssim {\norm{f}^2_{L^2_{q-1}(U)} + \norm{u}^2_{L^2_q(U)}}. \label{ec.1}
	\end{align}
	
	Since $ u\perp \ker(\Box_b) ${, closed range estimates imply that} $ \norm{u}^2_{L^2_q(U)}\lesssim  \norm{f}^2_{L^2_{q-1}} $. Thus, by \eqref{ec.1}, we obtain
	\[
	\sum_{J\in I_q}\sum_{k=1}^{2n-2} \norm{X_k u_J}^2_{L^2(\tilde V)}\lesssim \norm{f}^2_{L^2_{q-1}}.
	\]
	This last inequality and Proposition \ref{prop:2.7} give us
	\eqref{ineq3.2}. Similarly, we prove \eqref{ineq3.2p}.
\end{proof}

\begin{prop}\label{prop5.4}
	Fix $ U\subset\!\subset U^\prime $ and assume $ B(x,r)\subset U $.
	\begin{enumerate}[a)]
		\item\label{i:prop5.4a} If $ \dbarb u = (I-S_q^\prime)f $ in $ M $ with $ u\in C^\infty_{0,q-1}(M) $, $ f\in C^\infty_{0,q}(M) $, $ u\perp \ker \dbarb $ and $ \supp(f)\subset B(x,r) $, then
		\begin{equation}\label{prop5.4a}
		\norm{u}_{L^2_{q-1}}\leq Cr\norm{f}_{L^2_q(B(x,r))}.
		\end{equation}
		\item\label{i:prop5.4b} If $ \dbarbs u = (I-S_q)f $ in $ M $ with $ u\in C^\infty_{0,q+1}(M) $, $ f\in C^\infty_{0,q}(M) $, $ u\perp \ker \dbarbs $ and $ \supp(f)\subset B(x,r) $, then
		\[
		\norm{u}_{L^2_{q+1}}\leq Cr\norm{f}_{L^2_q(B(x,r))}.
		\] 
	\end{enumerate}
	Here $ C $ is a constant which does not depend on $ x $ and $ r $ (it can depend on $ U $ and $ U^\prime $).
\end{prop}

\begin{proof}
	We will prove $ a) $. The proof of $ b) $ is similar. Since $ u\perp \ker(\dbarb^{0,q-1}) $, a smooth $ (0,q) $-form $ \phi $ exists such that $ u=\dbarbs \phi $ and $ \phi\perp\ker(\dbarb^{*\ 0,q}) $ (for example, $ \phi=K_q\dbarb u $). By \eqref{prop5.3a} in Proposition \ref{prop:5.3}, we have
	$ \norm{\phi}_{L^2_q(B(x,r))}\leq Cr \norm{u}_{L^2_{q-1}}. $
	Let $ \ve>0 $, then
	\begin{align*}
		\norm{u}^2_{L^2_{q-1}}&
		= (f,(I-S_q^\prime)\phi)_{B(x,r)}=(f,\phi)_{B(x,r)} 
		\leq \norm{f}_{L^2_q}\norm{\phi}_{L^2_q(B(x,r))}\\
		&\leq \frac{r^2}{\ve}\norm{f}^2_{L^2_q}+\frac{\ve}{r^2}\norm{\phi}^2_{L^2_q(B(x,r))}
		\leq \frac{r^2}{\ve}\norm{f}^2_{L^2_q}+\ve C^2\norm{u}^2_{L^2_{q-1}}.
	\end{align*}
	We can absorb $ \norm{u}^2_{L^2_{q-1}} $ and obtain \eqref{prop5.4a} by choosing $ \ve $ small sufficient.
\end{proof}

	\begin{prop}\label{prop:5.5}
	Fix $ U \subset\!\subset U^\prime $ and let $ B(x,r)\subset U $. If $ f $ is a smooth $ (0,q) $-form with  $\supp(f)\subset B(x,r)$, then
	\begin{enumerate}[a)]
		\item\label{eq:f(2)}  $ \norm{G_qf}_{L^2_{q-1}}\leq Cr\norm{f}_{L^2_q(B(x,r))}$,
		\item\label{eq:f(4)}  $ \norm{G_q^\prime f}_{L^2_{q+1}}\leq Cr\norm{f}_{L^2_q(B(x,r))}$,
			\item\label{eq:c1} $\norm{(I-S_q^\prime)K_qf}_{L^2_q(B(x,r))}\leq Cr^2\norm{f}_{L^2_q(B(x,r))}$,
			\item\label{eq:c2} $\norm{(I-S_q)K_qf}_{L^2_q(B(x,r))}\leq Cr^2\norm{f}_{L^2_q(B(x,r))}$,
			\item\label{eq:c3} $\norm{K_qf}_{L^2_q(B(x,r))}\leq Cr^2\norm{f}_{L^2_q(B(x,r))}$.
	\end{enumerate}
\end{prop}

	\begin{proof}
		Item \ref{eq:f(2)}) follows by \eqref{i:prop5.4a} in Proposition \ref{prop5.4} and from the fact that $ \dbarb\dbarbs K_q=I-S_q^\prime $ and $ \Ran(\dbarbs K_q) \perp \ker{\dbarb} $. Similarly, we use \eqref{i:prop5.4b} in Proposition \ref{prop5.4} to obtain item \ref{eq:f(4)}) since $ \dbarbs\dbarb K_q=I-S_q $ and $ \Ran(\dbarb K_q) \perp \ker{\dbarbs} $. For the item \ref{eq:c1}), observe that since $ \dbarbs(I-S_q^\prime)K_q = \dbarbs K_q$ and $ (I-S_q^\prime)K_qf\perp \ker \dbarbs $, \eqref{prop5.3a} in Proposition \ref{prop:5.3} gives us
		\[
		\norm{(I-S_q^\prime)K_qf}_{L^2_q(B(x,r))} \leq cr\norm{\dbarbs K_q f}_{L^2_{q-1}}.
		\]
		So item \ref{eq:c1}) follows from item \ref{eq:f(2)}). 
		Since $ \dbarb(I-S_q)K_q=\dbarb K_q $ and $ (I-S_q)K_qf\perp \ker \dbarb $, \eqref{prop5.3b} in Proposition \ref{prop:5.3} yields
		\[
		\norm{(I-S_q)K_qf}_{L^2_q(B(x,r))} \leq cr\norm{\dbarb K_q f}_{L^2_{q+1}}.
		\]
		So, item \ref{eq:c2}) follows from item \ref{eq:f(4)}). Item \ref{eq:c3}) follows from items \ref{eq:c1}) and \ref{eq:c2}) and the fact that $ K_q=(I-S_q)K_q+(I-S_q^\prime)K_q $.
	\end{proof}

\begin{prop}\label{prop:ebfK}
	Under hypotheses of Theorem \ref{teo:5.1}, if $ \psi $ is a smooth function supported in $ B(x,r) $ normalized so that $ \norm{D^l\psi}_{L^\infty}\leq r^{-l} $ for $ 0\leq l\leq N $, then for every integer $ k\geq 0 $
	\begin{equation}\label{eq:ebfK1}
	\abs{D^k_zK_q^{IJ}(\psi)(z)}\leq C_k r^{2-k}\ , \qquad \forall z\in B(x,r)
	\end{equation}
	where $ D^k_z $ is of the type $ X_{j_1}\cdots X_{j_k} $ and $ N=N(k) $ is sufficiently large, for every multi {indices} $ I,J\in \I_q $. In fact, these estimates are satisfied for operators $ (S_qK_q)^{IJ} $ and $ (S_q'K_q)^{IJ} $.
\end{prop}

\begin{proof}
	Since $ K_q=S_qK_q+S_q'K_q $, it will be sufficient to prove \eqref{eq:ebfK1} for $ S_qK_q $ and $ S_q'K_q $ instead of $ K_q $. We will do this for $ S_qK_q $; the proof of $ S_q'K_q $ is similar. 
	Below, we will assume the scaled pullbacks following $ g_{x,R} $ for some $ R>0 $ {are} such that $ B(x,r) \subset \tilde{B}(x,R) $. Let $ \phi=\psi \bar{\omega}^J $, $ V\subset \widetilde{V}\subset B_0 $ be open sets, $ \zeta,\zeta^\prime,\zeta^{\prime\prime}\in C^\infty_0(B_0) $ such that $ \zeta\prec\zeta^\prime \prec \zeta^{\prime\prime}$ and $ \zeta\equiv 1 $ {on} $ V $. By Proposition \ref{prop:2.5} and Remark \ref{rem:pullback1}, we have
	\begin{align}
	\norm{(S_qK_q \phi)\hat{}\ }_{H^{s+\ve}_{q}(V)} &
	\lesssim \norm{\zeta' \hat{\p}_b (S_qK_q\phi)\hat{\ } }_{H_{q-1}^s(B_0)} + \abs{B(x,R)}^{-\frac{1}{2}} \norm{S_qK_q\phi }_{L^2_q(\tilde{B}(x,R))}. \label{eq:ebfK2}
	\end{align}

	By $ (a) $ and $ (b) $ in Theorem \ref{teo:2.4}, \eqref{eq:c2} and \eqref{eq:c3} in Proposition  \ref{prop:5.5}, we obtain
	\begin{align}
	\norm{S_qK_q \phi}_{L^2_q(\tilde{B}(x,R))}&
	\lesssim c_2^2R^2\norm{\phi}_{L^2_q(B(x,c_2R))}= c_2^2R^2\norm{\phi}_{L^2_q(\tilde{B}(x,R))} 
	\lesssim R^2 \abs{ B(x,R) }^{\frac{1}{2}}. \label{eq:ebfK3}
	\end{align}
	
	On the other hand, from \eqref{eq:k2.11} in Corollary \ref{prop:2.6} and using Remark \ref{rem:pullback1} again, we have
	\begin{align}
	\norm{ \zeta' \hat{\p}_b (S_qK_q\phi)\hat{\ } }_{H^s_{q-1}(B_0) } &\lesssim 
	R^2\norm{ \zeta'' (\dbarb \dbarbs S_qK_q\phi)\hat{\ } }_{H^{s-\ve/2}_{q}(B_0)} + R\norm{ (\dbarbs S_qK_q\phi)\hat{\ } }_{L^2_{q-1}(B_0)} \nn\\
	&\quad + \abs{B(x,R)}^{-\frac{1}{2}} \norm{S_qK_q\phi }_{L^2_q(\tilde{B}(x,R))}.
	\label{eq:ebfK4}
	\end{align}
	For $ N $ large sufficient, since $ \dbarb\dbarbs S_qK_q\phi=\dbarb\dbarbs K_q\phi=(I-S_q')\phi $, by Remark \ref{rem:ebfS6}, we obtain
	\begin{equation}
	\norm{\zeta''(\dbarb\dbarbs S_qK_q\phi)\hat{\ }}_{H^{s-\ve/2}_{q}(B_0)} = \norm{ \zeta''((I-S'_q)\phi)\hat{\ } }_{H^{s-\ve/2}_{q}(B_0)}\leq C_s . \label{eq:ebfK5}
	\end{equation}
	Also, since $ \dbarbs S_qK_q = \dbarbs K_q $, by Theorem \ref{teo:2.4} and \eqref{eq:f(2)} in Proposition \ref{prop:5.5}, we have
	\begin{align}
	\norm{ (\dbarbs S_qK_q\phi)\hat{\ }}_{L^2_{q-1}(B_0)} &
	\lesssim \abs{B(x,R) }^{-\frac{1}{2}}\norm{ \dbarbs K_q\phi }_{L^2_{q-1}(\tilde{B}(x,R))}  \nn\\
	&\leq c_2 R \abs{B(x,R) }^{-\frac{1}{2}}\norm{ \phi }_{L^2_q( B(x,c_2R))}  
	\lesssim R. \label{eq:ebfK6}
	\end{align}
	
	Taking \eqref{eq:ebfK3} -- \eqref{eq:ebfK6} into \eqref{eq:ebfK2}, we have
	\begin{equation*}
	\norm{ (S_qK_q\phi)\hat{\ } }_{H^{s+\ve}_{q}(V)} \leq C_s R^2.
	\end{equation*}
	Using {the Sobolev embedding theorem}, we obtain
	\begin{align*}
	\norm{\hat{D}^k_z (S_qK_q\phi)\hat{\ } }_{L^\infty_{q}({V})}\leq C_k R^2.
	\end{align*}
	Then
	\begin{align}
		\abs{D^k_z(S_qK_q)^{IJ}\psi(z)}&\leq \norm{D^k_zS_qK_q\phi}_{L^\infty_{q}(g_{x,R}(V))}
		= R^{-k}\norm{\hat{D}^k_z(S_qK_q\phi)\hat{\ }}_{L^\infty_{q}(V)}
		\leq C_kR^{2-k}, \label{eq:bfK10}
	\end{align}
	for all $ z
	$ such that $ z\in g_{x,R}(V) $. 
	Choosing $ R=r c_1^{-1} $ and $ V\subset B_0 $ such that $ B(x,r)\subset g_{x,R}(V) $, we have by $ a) $ in Theorem \ref{teo:2.4} that $ B(x,r)\subset \tilde{B}(x,R) $. Thus by \eqref{eq:bfK10}, it holds that
	\begin{align*}
	\abs{D^k_z(S_qK_q)^{IJ}\psi(z)}&\leq C_kr^{2-k}
	\end{align*}
	for all $ z\in B(x,r) $ and $ N $ sufficiently large.
\end{proof}

\begin{proof}[Proof of Theorem \ref{teo:5.1}]
	Since $ K_q=S_qK_q+S^\prime_qK_q $, it is sufficient to show the {result for the} operators $ (S_qK_q)^{IJ} $ and $ (S_q^{\prime}K_q)^{IJ} $. We will prove {the result for $ (S_qK_q)^{IJ} $ as t}he proof for $ (S_q^{\prime}K_q)^{IJ} $ is similar. 
	Let {$ c_0 $ be a positive constant such that $ r=c_0\rho(x,y) $ and $ \tilde{B}(x,r)\cap\tilde{B}(y,r) = \emptyset$, uniformly chosen for $ x,y $ and $r $}
	(we can choose $ c_0=(2c_2)^{-1} $ by Theorem \ref{teo:2.4}). Let $ f=\phi\bar{\omega}^J $, for $ J\in \I_q $, and $ \phi\in C^\infty_0(\tilde{B}(y,r)) $. For $ D^k $ as in the statement of the theorem, we affirm that the following holds
	\begin{align}\label{eq:203}
		\abs{ D^k((S_qK_q)^{IJ}\phi )(z)}&\leq C_k r^{2-k} \abs{ B(x,r) }^{-\frac{1}{2}} \norm{\phi}_{L^2}, \qquad \forall z\in \tilde{B}(x,r).
	\end{align}
	We prove \eqref{eq:203} as follows. Note that, because $ S_qK_q = (I-S_q^\prime)K_q$, 
	we have $ \dbarb\dbarbs S_qK_qf=(I-S_q^\prime)f $. Also $ \dbarb S_qK_qf=0  $. Doing a scaled pullback through $ g_{x,r} $, we obtain	
	\begin{align}\label{ec:200}
		\hat{\dbar}_b\hat{\p}_b (S_qK_qf)\hat{}=r^2((I-S_q^\prime)f)\hat{}\ , \qquad  \qquad \hat{\dbar}_b (S_qK_qf)\hat{}=0.
	\end{align}
	
	On the other hand, let $ C_1'>0 $ independent of $ x,y $ and $ r $ such that $ \tilde{B}(y,r)\subset B(x,C_1'r) $ and $ \tilde{B}(x,r)\subset B(x,C_1'r) $ (if $ c_0=(2c_2)^{-1} $, we can choose $ C_1'=4c_2 $). 	
	Then $ f $ will be supported on $ B(x, C_1'r) $. By Remark \ref{rem:pullback1} and \eqref{eq:c1} in Proposition \ref{prop:5.5}, we obtain
	\begin{align}
		\norm{(r^{-2}S_qK_qf)\hat{}\ }_{L^2_q(B_0)} &\lesssim \abs{ B(x,r) }^{-\frac{1}{2}}\norm{r^{-2}S_qK_qf}_{L^2_q(\tilde{B}(x,r))} 
		\leq \abs{ B(x,r) }^{-\frac{1}{2}}\norm{r^{-2}S_qK_q f}_{L^2_q(B(x,C_1'r))}  \nn\\
		&\leq  C\abs{ B(x,r) }^{-\frac{1}{2}}\norm{f}_{L^2_q} \label{ec:101} 
	\end{align}
	for some positive constant $ C $ independent of $ x,y $ and $ r $.  Also, from \eqref{eq:f(2)} in Proposition \ref{prop:5.5},
	and since 
	$ \dbarbs S_qK_q=G_q $, we obtain
	\begin{align}
		\norm{\hat{\p}_b (r^{-2}S_qK_qf)\hat{}\ }_{L^2_{q-1}(B_0)}&\leq r^{-1}\abs{B(x,r)}^{-\frac{1}{2}}\norm{\dbarbs S_qK_qf}_{L^2_{q-1}(\tilde{B}(x,r))} 
		\leq C\abs{B(x,r)}^{-\frac{1}{2}} \norm{f}_{L^2_q(B(x,C_1'r))} \nn \\		
		&= C\abs{B(x,r)}^{-\frac{1}{2}}\norm{f}_{L^2_q}. \label{ec:102}
	\end{align}

	Furthermore, since $ \dbarb (I-S^\prime_q)f=0 $, and $ \dbarbs(I-S_q')f=0 $ on $ \tilde{B}(x,r) $ because $ \supp(f)\subset \tilde{B}(y,r) $, using a pullback given by $ g_{x,r} $, we obtain
	\[
	\hat{\dbar}_b((I-S_q')f)\hat{} =0, \quad \text{ and } \quad \hat{\p}_b((I-S_q')f)\hat{} = 0 \quad \text{ on } B_0.
	\]
	Therefore, from Proposition \ref{prop:2.5}, we obtain for any $ \zeta\in C^\infty_0(B_0) $ that
	\begin{align}
		\norm{\zeta((I-S_q')f)\hat{}\ }_{H^s_{q}(B_0)}&\leq C_s \norm{((I-S_q')f)\hat{}\ }_{L^2_q(B_0)}
		\leq C_s \abs{B(x,r)}^{-\frac{1}{2}} \norm{(I-S_q')f}_{L^2_q(\tilde{B}(x,r))}\nn\\
		&\leq C_s\abs{ B(x,r) }^{-\frac{1}{2}}\norm{ f }_{L^2_q}\leq C_s\abs{ B(x,r) }^{-\frac{1}{2}}\norm{ f }_{L^2_q}. \label{ec:103}
	\end{align}
	The second inequality follows from Remark \ref{rem:pullback1} and Theorem \ref{teo:2.4}, and the third follows from the $ L^2 $ continuity of $ I-S_q' $.
	Using \eqref{eq:k2.12} in Corollary \ref{prop:2.6} and \eqref{ec:200}--\eqref{ec:103}, we obtain	
	\begin{align}\label{eq:201}
		\norm{\zeta'r^{-2}(S_qK_qf)\hat{}\ }_{H^s(B_0)}\leq  C_s \abs{B(x,r)}^{-\frac{1}{2}}\norm{f}_{L^2_q} .
	\end{align}
	Let $ V=g^{-1}_{x,r}(g_{x,r/2}(B_0)) $.
	Note that $ V\subset\!\subset B_0 $ is uniformly in $ x $ and $ r $.
	Considering $ \zeta\equiv 1 $ in $ V $ and $ s $ sufficiently large in \eqref{eq:201}, {the Sobolev embedding} theorem yields
	\begin{align}\label{eq:202}
	\norm{\hat{D}^k (S_qK_q f)\hat{}\ }_{L^\infty_{q}(V)}\leq C_k r^2\abs{B(x,r)}^{-\frac{1}{2}}\norm{f}_{L^2_q}
	\end{align}
	where $ C_k $ is independent of $ x $ and $ r $. Coming back to $ \tilde{B}(x,r) $ {via} $ g_{x,r} $, we obtain
	\begin{align}
		\abs{D^k(S_qK_q)^{IJ}\phi(z)}&\leq \norm{ D^k(S_qK_qf) }_{L^\infty_{q}(g_{x,r}(V))}=r^{-k}\norm{ \hat{D}^k(S_qK_q f)\hat{}\ } _{L^\infty_{q}(V)} \nn \\
		& \leq C_kr^{2-k}\abs{B(x,r)}^{-\frac{1}{2}}\norm{\phi}_{L^2}, 
		 \label{eq:204}
	\end{align}
	for all $ z\in g_{x,r}(V)=\tilde{B}(x,r/2) $, where $ D^k $ is of the {form} $ X_{j_1}X_{j_2}\cdots X_{j_k} $. As we did before, we can use Theorem \ref{teo:2.4} to obtain \eqref{eq:204} for $ z\in \tilde{B}(x,r) $. This {establishes} \eqref{eq:203}.

	Now we show the pointwise estimates for $ (S_qK_q)^{IJ} $ as a consequence of \eqref{eq:203}. Since \eqref{eq:203} is true for every $ \phi\in C^\infty_0(\tilde{B}(y,r)) $, $ L^2 $-duality will give us
	\begin{equation}\label{eq:4.3}
		\norm{D^k_z (S_qK_q)^{IJ}(z,\cdot)}_{L^2(\widetilde{B}(y,r)}\leq C_kr^{2-k}\abs{ B(x,r) }^{-\frac{1}{2}}
	\end{equation}
	for all $ z\in \widetilde{B}(x,r) $. Notice that the process we use to show \eqref{eq:203} can be reproduced if $ \phi\in C^\infty_0(\tilde{B}(x,r)) $ and we utilize $ g_{y,r} $ to scale the estimates. This, together with the fact that $ (S_qK_q)^{IJ}(z,w)=\ovl{ (S_qK_q)^{JI}(w,z) }$ because $S_qK_q $ is self-adjoint, allows us to conclude that
	\begin{equation}\label{eq:4.4}
		\qquad \norm{ D^l_w (S_qK_q)^{IJ} (\cdot, w)}_{L^2(\widetilde{B}(x,r))}\leq C_lr^{2-l} \abs{B(y,r)}^{-\frac{1}{2}},\qquad \forall w\in \widetilde{B}(y,r).
	\end{equation}
	Let $ \widehat{S_qK_q}{}^{IJ}(u,v):=(S_qK_q)^{IJ}(g_{x,r}(u),g_{y,r}(v)) $. By \eqref{eq:4.3}, \eqref{eq:4.4} and 	
	Theorem \ref{teo:2.4},  we have: 
	\begin{align}\label{eq:ee1}
		\norm{\hat{D}^k_u\widehat{S_qK_q}{}^{IJ}(u,\cdot)}_{L^2(B_0)} &\leq C^\prime_k r^2 \abs{B(x,r)}^{-\frac{1}{2}}\abs{B(y,r)}^{-\frac{1}{2}} \ ,\ \forall u\in B_0, \quad \text{ and }\\
		\norm{\hat{D}^l_v\widehat{S_qK_q}{}^{IJ}(\cdot,v)}_{L^2(B_0)} &\leq C^\prime_l r^2 \abs{B(x,r)}^{-\frac{1}{2}}\abs{B(y,r)}^{-\frac{1}{2}} \ ,\ \forall v\in B_0, \label{eq:ee2}
	\end{align}
	where $ \hat{D}^k_u $ is scaled through $ g_{x,r} $ and $ \hat{D}^l_v $ by $ g_{y,r} $.
	Since $ \abs{B(x,r)}\approx \abs{B(y,r)} $ uniformly
	{in} $ x,y$ and $ r $ because $ \rho(x,y)=cr $ and the doubling property; we have from \eqref{eq:ee1} and \eqref{eq:ee2}
	\begin{align*}
		\norm{\hat{D}^k_u \pare{ r^{-2}\abs{B(x,r)} \widehat{S_qK_q}{}^{IJ}(u,\cdot)}}_{L^2(B_0)} \leq C^\prime_k  \ ,\quad
		\norm{\hat{D}^l_v\pare{ r^{-2}\abs{B(x,r)} \widehat{S_qK_q}{}^{IJ}(\cdot,v)} }_{L^2(B_0)} \leq C^\prime_l , 
	\end{align*}
	for all $ (u,v)\in B_0\times B_0 $. Due to $ (e) $ and $ (f) $ in Theorem \ref{teo:2.4}, we can change $ \hat{D}^k_u $ and $ \hat{D}^l_v $ in these two last inequalities by $ \frac{\p^\alpha}{\p u^\alpha} $ and $ \frac{\p^\beta}{\p u^\beta} $ for multi indices $\alpha,\beta\in \N^{2n-1} $, respectively. Thus, a Fourier inversion argument gives us
	\[
	\nnorm{ \frac{\p^\alpha}{\p u^\alpha} \frac{\p^\beta}{\p v^\beta}\pare{ r^{-2}\abs{B(x,r)} \widehat{S_qK_q}{}^{IJ}(u,v)}}_{L^2(B_0\times B_0)} \leq C_{\alpha\beta}
	\]	
	uniformly in $ x $, $ y $ and $ r $. This, in particular, means that $ r^{-2}\abs{B(x,r)} \widehat{S_qK_q}{}^{IJ} \in C^\infty(B_0\times B_0)$, and with {the Sobolev} norms being uniformly bounded in $ x,y $ and $ r $.
	By $ (e) $ in Theorem \ref{teo:2.4} and {the Sobolev embedding theorem}, there will exist a positive constant $ C_{kl} $ such that
	\[
	\abs{ \hat{D}^k_u \hat{D}^l_v \pare{ r^{-2}\abs{B(x,r)} \widehat{S_qK_q}{}^{IJ}(u,v)} } \leq C_{kl}
	\]
	uniformly in $ x $, $ y $ and $ r $. Scaling back to $ \tilde{B}(x,r)\times \tilde{B}(y,r) $, we have
	\begin{equation*}
		\abs{ {D}^k_z {D}^l_w(S_qK_q)^{IJ}(z,w) } \leq C_{kl}^\prime r^{2-k-l}\abs{B(x,r)}^{-1} \ ,  \qquad \forall (z,w)\in \widetilde{B}(x,r) \times \widetilde{B}(y,r).
	\end{equation*}
	This proves the pointwise estimates \eqref{pointwisegreen} for $ (S_qK_q)^{IJ} $. The proof is complete.
\end{proof}

Below, we estate the result about {the} optimal gain of regularity for {the} Complex Green operator. By Theorem \ref{teo:5.1}, Proposition \ref{prop:ebfK} and since $ K_q: C_{0,q}^{\infty}(M)\to C_{0,q}^\infty(M) $, the proof follows exactly the same lines as \cite[Corollary 5.6]{koenig02}.

\begin{proposition}\label{prop:maximalestimates}
	Let $ M $ be a smooth, compact, orientable CR manifold of hypersurface type with closed range for $ \dbarb^{0,q-1} $ and $ \dbarb^{0,q} $. 
	Assume {the} comparable  weak $ Y(q) $ condition and finite type condition are satisfied in $ M $, 
	and let $ p\in\, ]1,\infty[ $. Then
	for each point in $ M $, there exists a neighborhood $ U $ such that
	\begin{equation}\label{ineq:lpmax}
		\sum_{i,j=1}^{2n-2}\norm{ X_iX_j u }_{L^p}\lesssim \norm{\Box_b u}_{L^p}+\norm{u}_{L^p}
	\end{equation}
	for every smooth $ (0,q) $-form $ u $ supported in $ U $.  (Here $ \lla{X_j}_{1\leq j\leq 2n-2} $ are real vector fields on $ U $ such that $ \lla{L_j=X_j+iX_{j+n-1}}_{1\leq j\leq n-1} $ span the CR structure on $ M $, and if $ u=\sum_{I\in\I_q}u_I\bar{\omega}^I $ then $ X_iX_ju=\sum_{I\in\I_q}(X_iX_ju_I)\bar{\omega}^I $).
\end{proposition}

\section*{Acknowledgement}
The author would like to express deep gratitude to A. Raich and T. Picon for generously sharing their expertise and insights. Their suggestions were invaluable in the creation of this article.

\bibliographystyle{abbrv}
\bibliography{bib}
\end{document}